\documentclass[a4, amsfonts]{amsart}

\usepackage{amssymb} 
\usepackage{amsthm}

\usepackage[colorlinks=true]{hyperref}
\usepackage{color}
\usepackage[all]{xy}
\usepackage{here}
\usepackage{tikz}

\usepackage{enumitem}

\newtheorem*{theorem*}{Theorem}
\newtheorem{theorem}{Theorem}[section]
\newtheorem{proposition}[theorem]{Proposition}
\newtheorem{lemma}[theorem]{Lemma}
\newtheorem{corollary}[theorem]{Corollary}

\theoremstyle{definition}
\newtheorem{definition}[theorem]{Definition}
\newtheorem{remark}[theorem]{Remark}
\newtheorem{example}[theorem]{Example} 
 
\newtheorem{question}[theorem]{Question}

\numberwithin{equation}{section}

\def\zQ{\mathbb Q}   
\def\zZ{\mathbb Z} 
\def\zN{\mathbb N} 

\newcounter{bean}

\newcommand{\Sq}{\operatorname{Sq}}
\newcommand{\Pe}{\mathcal{P}}

\newcommand{\Hom}{\mathrm{Hom}}
\newcommand{\SA}{\mathcal{A}}

\begin{document} 
\title{Graph colouring and Steenrod's problem for Stanley-Reisner rings}

\author[D. Stanley]{Donald Stanley}
\address{Department Mathematics and Statistics, University of Regina, Regina, Canada}
\email{Donald.Stanley@uregina.ca}

\author[M. Takeda]{Masahiro Takeda}
\address{Institute for Liberal Arts and Sciences, Kyoto University, Kyoto, 606-8316, Japan}
\email{takeda.masahiro.87u@kyoto-u.ac.jp}

\date{\today}



\maketitle

\begin{abstract}
   It is a classical problem in algebraic topology asked by Steenrod in \cite{S} which graded rings occurs as the cohomology ring of a space.
   In this paper, we define an algebraic version of the graph colouring, span colouring, and observe the relation between span colourings and Steenrod's problem for graded Stanley-Reisner rings, in other words polynomial rings divided by an ideal generated by square-free monic monomials.
\end{abstract}

\section{Introduction} 
\label{sec:intro} 

We are interested in which graded $\zZ$-algebras are realized as the cohomology of a space. We study this realization problem for Stanley-Reisner rings with generators in degrees $4$ and $6$.   We also construct span colouring, which is an approximation of colouring for graphs, and relate span colouring to the realization problem for these algebras. Details about terminology can be found in Section \ref{sec:notation}.

The general case of the realization problem, also known as Steenrod's problem, seems intractible though some special 
cases are known. For example the case $\zZ[a]/(a^3)$ is equivalent to the Hopf invariant one 
problem solved by Adams \cite{A}.
The polynomial ring case has been studied by
many researchers  (a very short and incomplete selection is \cite{AGMV,CE,N,ST}), and was finally completed by
Andersen and Grodal \cite{AG}. They showed a graded polynomial ring $\zZ[V]$ is realizable if and only if it is isomorphic to a tensor product of 
various copies of $H^*(BU(1))=\zZ[x_2]$, $H^*(BSU(n))\cong\zZ[x_4, x_6, \dots, x_{2n}]$ and 
$H^*(BSp(n))\cong\zZ[x_4, x_8, \dots, x_{4n}]$, where the subscripts correspond to the degrees of the generators.

We are interested in the larger class of Stanley-Riesner rings of a simplicial complex $K$. These are quotients of polynomial rings 
$SR(K,\phi)=\zZ[V]/I_K$, where the set of generators $V$ is the vertex set of $K$ and $I_K$ is the ideal generated by the non-faces of $K$. The extra $\phi$ in the notation is a map $\phi\colon V\rightarrow 2\zN $ that takes care of the degrees of the generators.

There is already some understanding of Steenrod's problem in this case. 
Davis and Januszkiewicz constructed realizing spaces for all $SR(K,\phi)$ with degree $2$ generators in \cite{DJ}. If the generators are of degree $2$ and $4$, then realizing spaces can be contructed by polyhedral products \cite{BBCG1,BBCG2}. 
However we know from the result for polynomial rings above there are severe restrictions on the possible degrees higher than $4$ even in the polynomial case. 
The second author gives realizablitiy conditions for a class of $SR(K, \phi)$ with general degrees \cite{Ta}. In particular in Takeda's class every pair of degree $4$ generators multiply to $0$. In other words there are no edges between vertices of degree $4$ in $K$.


In this paper we instead are more interested in  the opposite case where there are no relations 
between degree $4$ generators, and additionally we assume all generators have degrees $4$ and $6$. 
We particularly consider the rings $A(n,L)=SR(\Delta^{n-1}\ast L,\phi)$ depending on a natural number $n$ 
and a simplicial complex $L$ 
and relate the realization problem for this class of rings to colourings of the one skeleton of $L$. 
Before describing our realization results we first introduce span colourings and some of their properties. 

\subsection{Span colourings}

We consider graphs as simplicial complexes with no $2$-simplices, and we can also consider a graph 
in a more traditional way as a pair
\[
  G=(V,E)=(V(G), E(G))
\]
where $V$ is a set known as vertices, and $E\subset V\times V$ is the set of edges (in simplicial complex terms $E$ is the set of $1$-simplices). 
Let for $x\in V$, $N(x)$ denote the neighbours (or neighbourhood) of $x$ that is $N(x)=\{ y\in V | (x,y)\in E\}$.
Note that $N(x)$ doesn't contain $x$ itself.
Our first task in the paper will be to develop the theory of span colourings of graphs.

\begin{definition}[Weak span colouring] 
   Let $\mathbf{k}$ be a field and $G=(V,E)$ be a graph. 
   A  $n$-{\em weak span colouring} of $G$ is a function $f\colon V(G)\rightarrow \mathbf{k}^n\setminus \{0\}$ such that for every $x\in V$, $f(x)\notin \langle f(N(x))\rangle$, where $\langle S\rangle$ is the linear subspace spanned by 
   $S\subset \mathbf{k}^n$. 
\end{definition}
An ordinary colouring is a map $f\colon V(G)\rightarrow \{1,2,\dots n\}$ such that for every $x\in V$, $f(x)\notin f(N(x))$.
So weak span colouring is an algebraic version of graph colouring. In Section \ref{sec:span_coloring} we define two other related colourings, span colouring and intermediate span colouring, which we use for different purposes. The weak span colouring has the most naive definition. 
The least $n$ such that $G$ has an $n$-span colouring is the {\em span chromatic number} of $G$, denoted $s_{\mathbf{k}}\chi(G)$. Proposition \ref{same-chromatic-number} proves that the chromatic numbers associated to all three variants coincide. When $\mathbf{k}=\zZ/p$ we write $s_{\mathbf{k}}\chi=s_p\chi$.

For any colouring of any graph $G$ with $n$ colours we can also label the vertices by corresponding basis elements 
in $\mathbf{k}^n$ and any $n$-clique in $G$ needs at least $n$ linearly independent vectors to be span coloured so we get the following proposition, where $clique(G)$ and $\chi(G)$ are the clique and chromatic number of $G$. 
\begin{proposition}[Propositions \ref{important2} and \ref{important}]

For every field $\mathbf{k}$ and every graph $G$, $ clique(G)\leq s_{\mathbf{k}}\chi(G)\leq\chi(G)$.
\end{proposition}

If our field is finite we can construct graphs that exhibit an inequality. 
\begin{proposition}[Proposition \ref{chromatic_difference}]
   For any finite field $\mathbf{k}$, there is a graph $G$ such that
   \[
      s_{\mathbf{k}}\chi(G)<\chi(G).
   \]
\end{proposition}

To prove the last proposition we contruct a representing graph $A_{\mathbf{k}^n}$ for $n$-span colourings. 
This graph has the property that the $n$-span colouring of a graph $G$ are in bijection with 
$Hom(G,A_{\mathbf{k}^n})$. In this sense $K_n$ would be the representing graph for normal $n$-colourings. 
Since we have the identity map of $A_{\mathbf{k}^n}$ we see that $s_{\mathbf{k}}\chi(A_{\mathbf{k}^n})\leq n$ and 
by constructing $n$-cliques in $A_{\mathbf{k}^n}$ we get a lower bound and see that 
$s_{\mathbf{k}}\chi(A_{\mathbf{k}^n})=n$ .
Letting $q$ be the size of the finite field $\mathbf{k}$ and $p$ a prime, we show that for favorable combinations of $q$ and $p$, $Hom(A_{\mathbf{k}^p}, K_p)=\emptyset$ (Proposition \ref{cong}), and putting these together
$\chi(A_{\mathbf{k}^p})>p=s_{\mathbf{k}}\chi (A_{\mathbf{k}^p})$
(Example \ref{smallestexample}). 
There is also a characterisitic $0$ version of Proposition \ref{chromatic_difference}. There is a graph $G$
with  $\chi(G)=4$ and $s_{\mathbb{R}}\chi(G)=3$ (Example \ref{God}).

\subsection{Algebras with generators in degrees $4$ and $6$, $\bold{A(n,L)}$}

We will restrict to Stanley-Reisner ring $SR(K,\phi)$ of a simplicial complex $K$,  $\zZ[x_1, \dots x_n, y_1, \dots , y_m]/I_K$, generated by degree $4$ and $6$ generators $x_i$ and $y_i$. 
Consider the case $K$ a join $K=\Delta^{n-1}\ast L$ when the vertices of the $n-1$ simplex $\Delta^{n-1}$ have degree $4$ and the vertices of $L$ have degree $6$. 
Then we use the notation $A(n,L)=SR(K,\phi)$. This is one of the main objects of study in this paper. 

If $A(n,L)$ is realized by $X$, in other words $A(n,L)\cong H^*(X, \zZ)$, then $H^*(X, \zZ/2)$ must have an unstable action of the mod-$2$ Steenrod Algebra 
$\SA_2$. Since $A(n,L)$ is torsion free,  $H^*(X, \zZ/2)\cong A(n,L)\otimes \zZ/2$, and so we first look at when 
$A(n,L)\otimes \zZ/2$ has an action of 
$\SA_2$. 
Let $L^1$ be the $1$-skeleton of $L$. We can use an $n$-span colouring over $\zZ/2$ on $L^1$ to get an 
$\SA_2$ action on $A(n,L)\otimes \zZ/2$.

\begin{theorem}[Theorem \ref{construct_Steenrod_mod_str}]
Suppose $L$ is a simplicial complex such that 
    $s_2\chi(L^1)\leq n$. 
   Then $A(n,L) \otimes \zZ/2$ has an $\SA_2$ action. 
\end{theorem}

This theorem is proved by writing $A(n,L)\otimes \zZ/2$ as a limit of polynomial algebras and using the colouring to upgrade this limit to a limit in the category of unstable algebras over $\SA_2$. 

To prove the converse (Theorem \ref{Steenrod_alg_chromatic_number}) we need to assume that $L$ is equal to its $1$-skeleton, in other words that it is a graph. The proof constructs a span colouring out of an $\SA_2$ action. It requires some analysis of the $\SA_2$ action on the ideal $I_K$ and a reduction to graphs with minimal degree $2$ by sequentially removing degree $0$ and $1$ vertices.  Putting Theorems \ref{construct_Steenrod_mod_str} and \ref{Steenrod_alg_chromatic_number} together we get:

\begin{theorem}[Theorem \ref{sameSteen}]
 Suppose $G$ is a graph, then 
$A(n,G)\otimes \zZ/2$ has an $\SA_2$ action if and only if $s_2\chi(G)\leq n$. 
\end{theorem}

As observed above this means if $A(n,G)$ is realizable then $A(n,G)\otimes \zZ/2$ has an $\SA_2$ action and so Theorem \ref{Steenrod_alg_chromatic_number} also gives us:

\begin{theorem}[Corollary \ref{realimpspan}]
For any graph $G$, 
if $A(n,G)$ is realizable, then $G$ has an $n$-span colouring over $\zZ/2$ and so $s_2\chi(G)\leq n$. 
\end{theorem}

We can consider this as an inequality relating colourings to a topologically defined graph invariant. 
For a graph $G$ let $\chi_{Top}(G)$ be the least $n$ such that $A(n,G)$ is realizable. Then we can reinterpret the theorem as saying $s_2\chi(G)\leq \chi_{Top}(G)$. 
To get an upper bound we use the usual chromatic number $\chi$. 
If the chromatic number of $L^1$, the one skeleton of $L$, is less than $n$, then 
$A(n,L)$ is realizable.

 \begin{theorem}[Theorem \ref{color_theorem}]
 Let $L$ be a simplicial complex. 
 If $\chi(L^1)\leq n$ then $A(n,L)$ is realizable. 
 \end{theorem}
 
 The proof uses a category  $P_{\max}(K)$ and the colouring  to construct a functor of  spaces which are products of $BSU(2)$ and $BSU(3)$ and with colimit that realizes $A(n,L)$. Much of this process was worked out by the second author \cite{Ta}.  
 
 When our simplicial complex is a graph $G$ Corollary \ref{realimpspan} and Theorem \ref{color_theorem} together are equivalent to the following:

\begin{theorem}[Theorem \ref{topboundsfromgraphs}]
For any graph $G$, 
\begin{equation}\label{maininequ}
s_2\chi(G)\leq \chi_{Top}(G)\leq \chi(G)
\end{equation}
\end{theorem}

From Example \ref{smallestexample} we know $s_2\chi(A_{(\zZ/2)^3})<\chi(A_{(\zZ/2)^3})$, so at most one of the inequalities in Equation \ref{maininequ} can be an equality. To show $s_2\chi(G)=\chi_{Top}(G)$, or really any upper bound, would require the construction of a space. Although $A(n,G)$ can be written as a limit of realizable polynomial algebras, it does not seem possible to consistently realize all the maps and thus the whole diagram as a diagram of spaces. On the other hand showing $s_2\chi(G)<\chi_{Top}(G)$ requires finding more refined topological invariants that can detect the difference. These might give rise to combinatorially defined invariants of graphs that give approximations to $\chi$ that are better than $s_2\chi$. More discussion exists in Section \ref{sec:future_problem}. 

We also classify the realizable Stanley-Reisner algebras with two degree $4$ and an arbitrary number of degree $6$ generators (Theorem \ref{n_2_if_and_only_if_condition}).



\subsection{Outline of sections}

In Section \ref{sec:notation} we set some notation and give some information about the algebras $A(n,L)$. 
In Section \ref{sec:span_coloring} we introduce three versions of span colouring and show they all give rise to the same chromatic number. In Section \ref{sec:span_coloringgraph} we construct a graph $A_{\mathbf{k}^n}$ such that for any graph $G$, $Hom(G,A_{\mathbf{k}^n})$ is the set of $n$-span colourings of $G$. 
In Section \ref{sec:chromatic_number} we prove that for some graphs the span chromatic number is different from the ordinary chromatic number.
After this we consider the relation between colouring and Steenrod's problem for Stanley-Reisner 
rings.
In Section \ref{sec:Span_colouring_and_Steenrod_algebra}, we show that for any graph $G$ if $A(n,G)$ is realizable then the span chromatic number over $\zZ/2$ is at most $n$. 
In Section \ref{sec:Construction_of_Steenrod_algebra_structure}, we construct Steenrod algebra actions on Stanley-Reisner rings using span colourings. 
In Section \ref{sec:Proof_of_the_main_theorem} we construct spaces which realize $A(n,L)$ when $\chi(L^1)\leq n$ and prove the classification result in the case there are two degree $4$ generators. 
We give some general thoughts and ask questions in Section \ref{sec:future_problem}. 

\subsection{Acknowledgements}
The first author is supported by the Natural Sciences and Engineering Research Council of Canada grant RGPIN-05466-2020. The second author was supported by Foundation of Research Fellows, The Mathematical Society of Japan and
JSPS KAKENHI JP24KJ1758. 
Thanks to Karen Meagher for suggesting we use graph homomorphisms to study colourings. 
We would also like to thank Kyoto University where this work was begun and which supported the visit of the first author. 
The first author also thanks Taida Institute for Mathematical Sciences and Pacific Institute for Mathematical Sciences for their hospitality.

\section{Preliminaries}\label{sec:notation}

A simplicial complex on a vertex set $V$ is $K$ a subset of the power set of $V$ closed under inclusions. Note that this means $\emptyset$ is in $K$ except when $K$ is empty. 
For a simplicial complex $K$ its vertex set is denoted $V(K)$. 
So a simplex $\sigma\in K$ is a subset of $V(K)$.

We also use $\sigma$ to denote the power set of $\sigma$, $\mathcal{P} (\sigma)$ and consider it as the corresponding subsimplicial complex
$\sigma \subset K$. 
We let $\Delta^{n-1}$ the simplex with vertex set $\{ x_1, \dots , x_n\}$. 

We sometimes regard $K$ as a poset such that for $\sigma,\tau \in K$, $\sigma<\tau$ if and only if $\sigma\subset \tau$.
Let $P_{\max}(K)$ be a subposet of $K$ such that for $\sigma \in K$ $\sigma \in P_{\max}(K)$ if and only if there exist maximal simplices $\sigma_1, \dots \sigma_n \in K$ such that $\bigcap \sigma_i =\sigma$.
We regard $1$-skeleton of $K$ as a graph, and we denote this graph $K^1$, the same symbol we use for the $1$-skeleton of $K$.

For a subset $\sigma\subset V(K)$, let $K_{\sigma}$ denote the full subcomplex of $K$ with vertex set $\sigma$. We can write it as a projection or as a subcomplex

$$L_{\sigma}=\{ \tau\cap \sigma | \tau\in L\}=\{\tau | \tau \in L, \tau\subset \sigma\}$$

For $K,L$ simplicial complexes the join $K\ast L=\{ \sigma \cup \tau | \sigma\in K, \tau\in L\}$ is a simplicial 
complex on $V(K)\cup V(L)$. 

The poset of $K\ast L$ is the product of the posets of $K$ and $L$. 

Consider the geometric realization of a simplicial complex $K$, $|K|$. Although we will not use realization it is compatible with joins, 
meaning $|K\ast L|\cong |K| \ast |L|$.

We consider graphs as simplicial complexes with only $0$ and $1$-simplices. 
We also use more standard notation for graphs and let $G=(V,E)$ denote a graph with vertex set $V$ and edge set $E$.

For  $x\in V$ the neighbourhood (or set of neighbours) of $x$ is $N(x)=\{ y| (x,y)\in E\}$.

$K_n$ denotes the complete graph on vertex set $\{1, \cdots , n\}$. 
An $n$-clique in $G$ is a subgraph isomorphic to $K_n$. 
The clique number of $G$, $clique(G)$ is the size of the largest clique in $G$.

We let $\mathbf{k}$ denote a field and 
$\mathbf{k}^n$ be the $\mathbf{k}$-vector space with basis $\{ e_1, \cdots, e_n\}$. Recall that $\langle S\rangle$ denotes the linear subspace spanned by $S\subset \mathbf{k}^n$, and so 
$\mathbf{k}^n=\langle  e_1, \cdots, e_n\rangle$. 
Let $ P\mathbf{k}^n$ denote the projective space of lines in $\mathbf{k}^n$, an element of $P\mathbf{k}^n$ is a dimension $1$ subspace of $\mathbf{k}^n$.
Recall $\mathrm{Gr}_l(\mathbf{k}^n)$ denotes the Grassmannian, more explicitly it is the set of $l$ dimensional subspaces of $\mathbf{k}^n$. Using this notation $P\mathbf{k}^n=\mathrm{Gr}_1(\mathbf{k}^n)$.

Let $K$ be a simplicial complex with the vertex set $V$, and $\phi \colon V\rightarrow 2\zZ_{>0}$.
Then the graded Stanley-Reisner ring $SR(L,\phi)$ is defined by
\[
   SR(K,\phi)= \zZ[V]/I_K,
\] 
where $\zZ[V]$ is the polynomial ring generated by $x \in V$ with $|x|=\phi(x)$ and $I_K$ is the ideal generated by monomials $x_1 x_2\dots x_k$ with $\{ x_1,x_2,\dots ,x_k\}\notin K$ as a simplex.

We need a functoriality property of $SR(K,\phi)$. 
If we have a subcomplex $K'\subset K$ then we get an algebra map 
$\pi\colon SR(K,\phi)\rightarrow SR(K', \phi)$, since $I_K\subset I_{K'}$.

For example $\sigma\in K$ can be considered as a simplicial complex and we get a map 
$SR(K,\phi)\rightarrow SR(\sigma,\phi)=\zZ[\sigma]$.

\begin{example}
$SR(K\ast L, \phi_K\cup \phi_L)\cong SR(K, \phi_K)\otimes SR(L, \phi_L)$. This follows since the minimal missing simplices of $K\ast L$ consists of the union of the minimal missing simplices of $K$ and $L$. 
\end{example}

\subsection{The algebras $A(n, L)$ and $A(n,G)$}
 
 We will restict to looking at the case when all generators are of degrees $4$ and $6$. 
 
Recall the definition of $\Delta^{n-1}$ above. The realizability of the algebras $A(n,L)$ and $A(n,G)$ is the main problem studied in the second half of the paper. 
 
 \begin{definition}
 
Suppose $L$ is a simplicial complex with $V(L)=\{ y_1, \dots,  y_m \}$ and $n$ a non-negative integer. Let $A(n,L)=SR(\Delta^{n-1}\ast L, \phi)$ where $\phi(x_i)=4$ and $\phi(y_i)=6$. 

We sometimes particularly consider the case $A(n,G)$ where $G$ is a graph considered as a simplicial complex. 

\end{definition}

We record the following lemma without proof. 

\begin{lemma}\label{46abthings}
Suppose we have a Stanley-Reisner algebra whose generators are in degrees $4$ and $6$, 

$$SR(K,\phi)=\zZ[ x_1, \dots, x_n,  y_1, \dots,  y_m]/I_K$$

where $\phi(x_i)=4$ and $\phi(y_i)=6$. 

\begin{enumerate}
\item Then there is a simplicial complex $L$ such that $SR(K,\phi)=A(n,L)$
if and only if $I_K$ is generated by products of elements $y_j$. In other words there are no relations between the $x_i$ or between the $x_i$ and $y_j$. 
\item There is a graph $G$ such that  $SR(K,\phi)=A(n,G)$ if and only if in addition to (1) for every $i<j<k$, $y_iy_jy_k\in I_K$. 
\end{enumerate}
\end{lemma}

\section{Roughly equivalent versions of span colouring}\label{sec:span_coloring}

In this section we fix a graph $G=(V,E)$  and look at some basic properties of three slightly different versions of span colouring of a graph which all depend on our field $\mathbf{k}$. All three give the same associated chromatic number (Proposition \ref{same-chromatic-number}) and are in that sense roughly equivalent. The first version of the span colouring has already been defined in Section \ref{sec:intro}. It is the most naive and easy to understand.  Recall the set of neighbours of $x$, $N(x)$ from Section \ref{sec:notation}

\begin{definition}[Weak span colouring] 
   A  $n$-{\em weak span colouring} of $G$ over $\mathbf{k}$ is a function $f\colon V(G)\rightarrow \mathbf{k}^n\setminus\{0\}$ such that for every $x\in V$, $f(x)\notin \langle f(N(x)) \rangle$.
      
   The {\em weak span chromatic number} of $G$, $s''_{\mathbf{k}}\chi(G)$ is the least $n$ such that $G$ has an $n$-weak span colouring.
\end{definition}

 The second version is similar to the first but instead uses the projective space. 

\begin{definition}[Intermediate Span colouring] 
   A $n$-{\em intermediate span colouring} of $G$ over $\mathbf{k}$ is a function $f\colon G\rightarrow P\mathbf{k}^n$ such that for every $x\in V$, $f(x)\not\subset \langle f(N(x)) \rangle$.
   
   The {\em intermediate span chromatic number} of $G$, $s'_{\mathbf{k}}\chi(G)$ is the least $n$ such that $G$ has an $n$-intermediate span colouring.
\end{definition}

Note that if $\mathbf{k}=\zZ/2$ these two versions of Span colouring are the same since $\zZ/2^n\setminus\{0\}\rightarrow P\zZ/2^n$ is a bijection. The final version of span colouring is a souped up version of the second which will allow us to construct a graph $A_{\mathbf{k}^n}$ such that $Hom(G,A_{\mathbf{k}^n})$ is bijective with $n$-span colourings of $G$ (see Proposition \ref{span_colouring_graph_hom}). On the surface it is more technical. 

\begin{definition}[Span colouring] 
   An $n$-{\em span colouring} of $G$ over $\mathbf{k}$ is a function $f\colon G\rightarrow P\mathbf{k}^n\times \mathrm{Gr}_{n-1}(\mathbf{k}^n)$  such that for $x\in V$ with $f(x)=(U,V)$, $U\not\subset V$, and for every $(x,y)\in E$, with $f(y)=(U',V')$, $U\subset V'$ and $U'\subset V$.

   The {\em span chromatic number} of $G$, $s_{\mathbf{k}}\chi(G)$ is the least $n$ such that $G$ has an $n$-span colouring. When $\mathbf{k}=\zZ/p$ we write $s_{\mathbf{k}}\chi=s_p\chi$. 
\end{definition}

We will show using a sequence of lemmas that all three chromatic numbers are the same. 

\begin{lemma}\label{span-vs-intermediate}
Suppose $f\colon V\rightarrow  P\mathbf{k}^n\times \mathrm{Gr}_{n-1}(\mathbf{k}^n)$ is an $n$-span colouring. Then $p_1\circ f \colon V\rightarrow P\mathbf{k}^n\times \mathrm{Gr}_{n-1}(\mathbf{k}^n) \rightarrow P\mathbf{k}^n$, 
where $p_1$ is projection onto the first factor, is an $n$-intermediate span colouring. Hence $s'_{\mathbf{k}}\chi(G)\leq s_{\mathbf{k}}\chi(G)$. 
\end{lemma}

\begin{proof}
Let $p_2\colon P\mathbf{k}^n\times \mathrm{Gr}_{n-1}(\mathbf{k}^n) \rightarrow \mathrm{Gr}_{n-1}(\mathbf{k}^n)$ be projection into the second factor. For $x\in V$, if 
$y\in N(x)$, then $p_1\circ f(y)\subset p_2f(x)$. Hence $\langle p_1f(N(x))\rangle \subset p_2f(x)$. So $p_1f(x)\not\subset \langle p_1f(N(x))\rangle$ and 
$p_1f$ is an intermediate span colouring. 
\end{proof}

\begin{lemma}\label{intermediate-vs-span}
Suppose $f\colon V\rightarrow  P\mathbf{k}^n$ is an $n$-intermediate span colouring. 
Then  there exists an $n$-span colouring $g\colon V\rightarrow  P\mathbf{k}^n\times \mathrm{Gr}_{n-1}(\mathbf{k}^n)$ such that $p_1\circ g=f$. 
Hence $s_{\mathbf{k}}\chi(G)\leq s'_{\mathbf{k}}\chi(G)$. 
\end{lemma}

\begin{proof}
For $x\in V$, let $U_x\in \mathrm{Gr}_{n-1}(\mathbf{k}^n)$ be any $n-1$ dimensional subspace such that  $\langle f(N(x)) \rangle\subset U_x$ and $f(x)\not\subset U_x$, and $g\colon V\rightarrow  P\mathbf{k}^n\times \mathrm{Gr}_{n-1}(\mathbf{k}^n)$ be given by $g(x)=(f(x), U_x)$. 
Then $g$ is an $n$-span colouring.  To see this 
let $x\in V$, and we just need to check that for $y\in N(x)$, $f(y)\subset U_x$. This follows since $f(y)\subset \langle f(N(x)) \rangle\subset U_x$ and $g$ is an $n$-span colouring. 
\end{proof}

The proofs of the next two lemmas we leave as straightforward checks for the reader. 

\begin{lemma}\label{weak-vs-intermediate}
Suppose $f\colon V\rightarrow \mathbf{k}^n\setminus\{0\}$  is an $n$-weak span colouring. Let $g\colon  V\rightarrow  P\mathbf{k}^n$ be given by $g(x)=\langle f(x)\rangle$. 
Then $g$ is an $n$-intermediate span colouring. 
Hence $s'_{\mathbf{k}}\chi(G)\leq s''_{\mathbf{k}}\chi(G)$. 
\end{lemma}

\begin{lemma}\label{intermediate-vs-weak}
Suppose $f\colon V\rightarrow P\mathbf{k}^n$ is an $n$-internediate span colouring. For each $x\in V$, pick $v_x\in \mathbf{k}^n$ such that 
$f(x)=\langle v\rangle$, and let $g\colon  V\rightarrow \mathbf{k}^n\setminus\{0\} $ be given by $g(x)=v_x$. Then 
$g$ is an $n$-weak span colouring. 
 Hence $s''_{\mathbf{k}}\chi(G)\leq s'_{\mathbf{k}}\chi(G)$. 
 \end{lemma}
 
 The last four lemmas give us, 

\begin{proposition}\label{same-chromatic-number}
$s_{\mathbf{k}}\chi(G)=s'_{\mathbf{k}}\chi(G)=s''_{\mathbf{k}}\chi(G)$.
\end{proposition}

Having proved the last proposition we only use the notation $s_{\mathbf{k}}\chi$ and discard the other two notations for (what we now know are all) span chromatic number.  Although the chomatic numbers are the same, the number of colourings of each type is slightly different.

\begin{remark}
From Lemma \ref{intermediate-vs-weak} we see that to get from an intermediate span colouring 
to a weak span colouring for each vertex we have to make a choice of a generator for the 
one dimensional subspace corresponding that vertex. Suppose $\mathbf{k}$ is finite 
with  $q$ elements, then there are $(q-1)^{|V|}$ weak span colourings for every intermediate span colouring. 

However the relation between intermediate span colourings and span colourings is more complicated. Lemma \ref{intermediate-vs-span} tells us that to get a span colouring from an intermediate one we have to extend 
$\langle f(N(x))\rangle$ to a certain kind of $n-1$ dimensional subspace. The number of ways of doing this only depends on all the dimensions of the $\langle f(N(x))\rangle$, and these might depend on the particular intermediate colouring. 

\begin{figure}[H]
   \begin{tabular}{cc}
      \begin{minipage}[t]{0.45\hsize}
         \begin{center}
            \begin{tikzpicture}[x=1cm, y=1cm, thick]
            \fill(0:1.9) circle[radius=2.5pt];
            \fill(60:1.9) circle[radius=2.5pt];
            \fill(120:1.9) circle[radius=2.5pt];
            \fill(180:1.9) circle[radius=2.5pt];
            \fill(240:1.9) circle[radius=2.5pt];
            \fill(300:1.9) circle[radius=2.5pt];
    
            \draw(0:1.9)--(60:1.9)--(120:1.9)--(180:1.9)--(240:1.9)--(300:1.9)--(0:1.9);
            \draw(120:1.9)--(240:1.9);

            \draw(0:1.9) node[right]{$\langle e_2\rangle$};
            \draw(60:1.9) node[right]{$\langle e_1\rangle$};
            \draw(120:1.9) node[left]{$\langle e_2\rangle$};
            \draw(180:1.9) node[left]{$\langle e_1\rangle$};
            \draw(240:1.9) node[left]{$\langle e_3\rangle$};
            \draw(300:1.9) node[right]{$\langle e_1\rangle$};
         \end{tikzpicture}
         \caption{}
      \end{center}
      \end{minipage}
      \begin{minipage}[t]{0.45\hsize}
         \begin{center}
         \begin{tikzpicture}[x=1cm, y=1cm, thick]
            \fill(0:1.9) circle[radius=2.5pt];
            \fill(60:1.9) circle[radius=2.5pt];
            \fill(120:1.9) circle[radius=2.5pt];
            \fill(180:1.9) circle[radius=2.5pt];
            \fill(240:1.9) circle[radius=2.5pt];
            \fill(300:1.9) circle[radius=2.5pt];
    
            \draw(0:1.9)--(60:1.9)--(120:1.9)--(180:1.9)--(240:1.9)--(300:1.9)--(0:1.9);
            \draw(120:1.9)--(240:1.9);

            \draw(0:1.9) node[right]{$\langle e_1\rangle$};
            \draw(60:1.9) node[right]{$\langle e_3\rangle$};
            \draw(120:1.9) node[left]{$\langle e_2\rangle$};
            \draw(180:1.9) node[left]{$\langle e_1\rangle$};
            \draw(240:1.9) node[left]{$\langle e_3\rangle$};
            \draw(300:1.9) node[right]{$\langle e_2\rangle$};
         \end{tikzpicture}
         \caption{}
      \end{center}
      \end{minipage}
   \end{tabular}
\end{figure}

Consider the graphs above with the line associated to each vertex being the subspace written beside the vertex. This gives two different  $3$-intermediate span colourings of a graph where in the first colouring $\dim \langle f(N(x))\rangle=2$ for each vertex  while in the second colouring there are two vertices with $\dim \langle f(N(x))\rangle=1$. Thus we see there is exactly one $3$-span colouring corresponding to the first intermediate colouring but $q^2$,
$3$-span colourings corresponding to the second. The number of ways to extend a one dimensional subspace in $\mathbf{k}^3$ to a two dimensional one while still avoiding non-trivial intersection with another one dimensional subspace is $q$, so making such a choice for two vertices gives $q^2$. 

For all intermediate span colourings of a graph satisfying
$\dim \langle f(N(x))\rangle=s_{\mathbf{k}}\chi(G)-1$ for each $x\in V$ seems like some kind of regularity property of the graph. For example could it be that vertex transitive graphs have this property. Of course these properties could also depend on the field. 
\end{remark}

\section{A graph representing span colourings}\label{sec:span_coloringgraph}

Again for this section we fix a graph $G=(V,E)$ and now we also fix a field $\mathbf{k}$. We construct a graph $A_{\mathbf{k}^n}$ which represents $n$-span colourings. In other words such that $\Hom( G, A_{\mathbf{k}^n})$ is in bijection with the $n$-span colourings of $G$. 

Let $K_n$ denote the complete graph on the vertex set $[n]$.
Given a map $f\colon G \rightarrow K_n$, the map $V(f)\colon V(G)\rightarrow V(K_n)=[n]$ is an $n$-colouring of $G$ and also $\theta\colon \Hom(G,K_n)\rightarrow \{Colourings \ of \ G\} $ defined by $\theta(f)=V(f)$ is a bijection. 
We would like to construct a similar graph for span colourings. 

We define a graph $A_{\mathbf{k}^n}$ as follows;
\begin{align*}
   V(A_{\mathbf{k}^n})=&\{(U,V)\in P\mathbf{k}^n\times \mathrm{Gr}_{n-1}(\mathbf{k}^n)| U\not\subset V\},\\
   E(A_{\mathbf{k}^n})=&\{((U,V), (U',V'))| U\subset V', U'\subset V\}.
\end{align*} 
Given a map $f\colon G\rightarrow A_{\mathbf{k}^n}$ for $v\in G$ we can write $f(v)=(U_v, V_v)$. 
Then it is easy to check that $f\colon V\rightarrow P\mathbf{k}^n\times \mathrm{Gr}_{n-1}(\mathbf{k}^n)$, 
is an $n$-span colouring of $G$. 
As before we can define $\theta\colon \Hom(G,A_{\mathbf{k}^n})\rightarrow \{n\text{-span colourings of } G\}$, 
$\phi(f)=V(f)$.

\begin{remark}
The complement of $A_{\mathbf{k}^n}$ in  $P\mathbf{k}^n\times \mathrm{Gr}_{n-1}(\mathbf{k}^n)$ is the partial flag manifold $Fl=\{ U\subset V \subset \mathbf{k}^n | dim(U)=1, dim(V)=n-1\}$. This is 
$Fl= (P\mathbf{k}^n\times \mathrm{Gr}_{n-1}(\mathbf{k}^n))\setminus A_{\mathbf{k}^n}$. When $\mathbf{k}$ is finite this gives the equation $|Fl|+ |A_{\mathbf{k}^n}|=|(P\mathbf{k}^n\times \mathrm{Gr}_{n-1}(\mathbf{k}^n))|$.

\end{remark}

The following two results follow directly.

\begin{proposition}\label{span_colouring_graph_hom}
   $\phi\colon \Hom(G,A_{\mathbf{k}^n})\rightarrow \{n\text{-span colourings of } G\}$ is a bijection. 
\end{proposition}


\begin{corollary}\label{weak_span_colouring_graph_hom}
   $s_\mathbf{k}\chi(G)\leq n$ if and only if $\Hom(G,A_{\mathbf{k}^n})\neq \emptyset$. 
\end{corollary}

   We observe some basic properties about the chromatic numbers.
The following lemma follows directly from the definition of $A_{\mathbf{k}^n}$.
\begin{lemma}\label{clique_A_k}
Let $\{a_1, \cdots , a_n\}$ be a basis for $\mathbf{k}^n$. 
Then $\{ (\langle a_j \rangle, \langle a_1, \cdots \widehat{a_j} \cdots,  a_n \rangle)| 1\leq j \leq n  \}$ is an n-clique in $A_{\mathbf{k}^n}$. 
\end{lemma}

Observe that the $n$-clique of the last lemma only depends on the lines containing the $a_i$. Then 
the clique given by  $\{a_1, \cdots , a_n\}$ determines an unordered configuration of $n$ points in $P\mathbf{k}^n$. 
The list of cliques in the lemma will be those configurations that span $\mathbf{k}^n$. 
Put another way if we start out with an ordered basis we can quotient out by free
$(\mathbf{k}^*)^n$ and symmetric group actions to we get the set of 
n-cliques as a subset of the configuration space. We will use the diagonal part of the free $(\mathbf{k}^*)^n$ action in the next section.

\begin{lemma}
Suppose $\{(U_1,V_1), \dots, (U_k, V_k)\}$ is a clique in $A_{\mathbf{k}^n}$. Then
$dim \langle U_1, \dots, U_k\rangle=k$.  If $k=n$ then $V_i=\langle U_1 \dots, \widehat{U_i}, \cdots,  U_n\rangle$.  
\end{lemma}

\begin{proof}
Starting with the first part $U_i\not\in \langle U_1,  \cdots \widehat{U_i} \cdots,  U_k \rangle$ so the 
$U_i$ are linearly independent so $\dim\langle (U_1,V_1), \dots, (U_k, V_k)\rangle=k $

Using the first part for the second part 
$\langle U_1,  \cdots \widehat{U_i} \cdots,  U_k \rangle\subset V_i$ and both have dimension $n-1$ so they are equal. 
\end{proof}

The last two lemmas show the next two.

\begin{proposition}\label{cliqueAn}
$clique(A_{\mathbf{k}^n})=n$. 
\end{proposition}

\begin{proposition}
All $n$-cliques in $A_{\mathbf{k}^n}$ are of the form in Lemma \ref{clique_A_k}. 
\end{proposition}

\begin{proposition}\label{important2}
   For any graph $G$, $clique(G)\leq s_{\mathbf{k}}\chi(G)$. 
\end{proposition}
\begin{proof}
Follows since $clique(A_{\mathbf{k}^n})=n$. If $ s_{\mathbf{k}}\chi(G)=n$ then there is a map 
$f\colon G\rightarrow A_{\mathbf{k}^n}$ so $f$ of any clique in $G$ is a clique in $A_{\mathbf{k}^n}$ so 
$clique(G)\leq clique(A_{\mathbf{k}^n})=n$. 
\end{proof}

\begin{proposition}\label{numberofAn}
$s_{\mathbf{k}}\chi(A_{\mathbf{k}^n})=n$. 
\end{proposition}

\begin{proof}
This follows from Propositions \ref{cliqueAn} and \ref{important2} and Corollary \ref{weak_span_colouring_graph_hom}.
\end{proof}

\begin{proposition}\label{important}
   For any graph $G$, $s_{\mathbf{k}}\chi(G)\leq \chi(G)$. 
\end{proposition}
\begin{proof} We describe two ways of showing this. 

   Given a colouring $f\colon G\rightarrow [n]$ we can consider it as a map $g\colon G\rightarrow P\mathbf{k}^n$ by taking $g(v)$ the one dimensional subspace generated by the basis element $e_{f(v)}$. 
   This gives an intermediate $n$-span colouring. Finally apply Proposition \ref{same-chromatic-number}.

   If $\chi(G)=n$ there is a graph homomorphism $G\rightarrow K_n$, and by Lemma \ref{clique_A_k}, there is a graph homomorphism $K_n\rightarrow A_{\mathbf{k}^n}$. Composing we get $G\rightarrow  A_{\mathbf{k}^n}$ and
   thus by Lemma \ref{span_colouring_graph_hom}, $s_{\mathbf{k}}\chi(G)\leq n$. 
\end{proof}
These two chromatic numbers are equal for small values. Note that $K_0=A_{\mathbf{k}^0}=\emptyset$ is the null graph and 
$K_1=A_{\mathbf{k}^1}$ is one point. Thus only the null graph has chromatic number or span chromatic number $0$, and only discrete graphs have chromatic or span chromatic number $1$. Lets see what happens when the chromatic numbers are $2$. 
\begin{proposition}\label{2_dim_chromatic}
    For any field $\mathbf{k}$, $\chi(A_{\mathbf{k}^2})=2$ and so $\chi(G)=2$ if and only if $s_\mathbf{k} \chi(G)=2$.
\end{proposition}
\begin{proof}
   For any field each vertex of $A_{\mathbf{k}^2}$ has degree $1$ with $(U,V)$ connected to $(V,U)$. 
   The implies $A_{\mathbf{k}^2}$ is a disjoint union of $K_2$. Hence $\chi(A_{\mathbf{k}^2})=2$. The second statement follows by arguing as in the second proof of the last proposition.
   
\end{proof}
\begin{example}
   The graph $A_{(\zZ/2)^2}$ is the following.
   
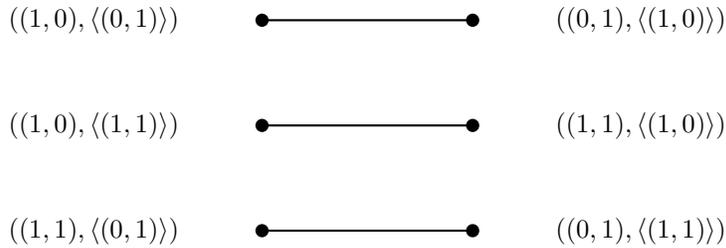
\begin{figure}[H]
      {
      \begin{tikzpicture}[x=0.7cm, y=0.7cm, thick]
        \draw(0,0)--(4,0);
        \draw(0,2)--(4,2);
        \draw(0,4)--(4,4);
        \fill(0,0) circle[radius=2.5pt];
        \fill(0,2) circle[radius=2.5pt];
        \fill(0,4) circle[radius=2.5pt];
        \fill(4,0) circle[radius=2.5pt];
        \fill(4,2) circle[radius=2.5pt];
        \fill(4,4) circle[radius=2.5pt];
        \draw(-5,4) node[right]{$((1,0),\langle(0,1)\rangle)$};
        \draw(9,4) node[left]{$((0,1),\langle(1,0)\rangle)$};
        \draw(-5,2) node[right]{$((1,0),\langle(1,1)\rangle)$};
        \draw(9,2) node[left]{$((1,1),\langle(1,0)\rangle)$};
        \draw(-5,0) node[right]{$((1,1),\langle(0,1)\rangle)$};
        \draw(9,0) node[left]{$((0,1),\langle(1,1)\rangle)$};
      \end{tikzpicture}}
    \end{figure}
\end{example}

Note that the clique number of an odd cycle is $2$, but its chromatic and span-chromatic numbers are $3$. 
In general the two chromatic numbers $\chi$ and $s_\mathbf{k} \chi$ do not coincide.
We will prove this in the next section.







\section{Chromatic number of $A_{\mathbf{k}^n}$}\label{sec:chromatic_number}

For this section we let $\mathbf{k}$ be a finite field of cardinality $q$.

For a basis $B=\{a_1, \cdots , a_n\}$  of $\mathbf{k}^n$, we let $f_B\colon K_n\rightarrow A_{\mathbf{k}^n}$ be the map corresponding to the clique of Lemma \ref{clique_A_k}.

Let $\mathcal B$ be the set of bases of $\mathbf{k}^n$and $F'=\coprod_{B\in \mathcal{B}} f_B \colon \coprod_{B\in \mathcal{B}} K_n \rightarrow A_{\mathbf{k}^n}$.

Notice that $\mathcal B$ has a free action of $\mathbf{k}^*$ given by coordinate wise multiplication, 
$\alpha\cdot \{a_1, \cdots , a_n\}=\{\alpha a_1, \cdots , \alpha a_n\}$, and let $\overline{\mathcal B}$ denote the quotient. 
    Since $f_B(\alpha\cdot B)=f_B(B)$ we get an extension of $F'$ to 
    $F=\coprod_{B\in \overline{\mathcal{B}}}\overline{f}_B \colon \coprod_{B\in \overline{\mathcal{B}}} K_n \rightarrow A_{\mathbf{k}^n}$. 
    Note there is a larger free action of $(\mathbf{k}^*)^n$ on $\mathcal B$, but we quotient out by the smaller actions since it makes the formulas slightly easier. Quotienting out by the larger action may give better results.

Next something well known, 

\begin{lemma}\label{basis_count}
$|\mathcal{B}|=\frac{1}{n!}\prod_{i=0}^{n-1}(q^{n}-q^i)$
\end{lemma}

\begin{proof}
Ordered bases are in bijection with $n$ by $n$ matrices with non-zero determinant this 
gives $\prod_{i=0}^{n-1}(q^{n}-q^i)$, next going from ordered to unordered divides the result by $n!$. 
\end{proof}

\begin{lemma}\label{inverseAn}
For any $(U,V)\in V(A_{\mathbf{k}^n})$, $|F^{-1}((U,V))|=\frac{1}{(n-1)!}\prod_{i=0}^{n-2}(q^{n-1}-q^i)$
\end{lemma}

\begin{proof}
Given a $1$-dimensional subspace $U\subset  \mathbf{k}^n$ and an $(n-1)$-dimensional subspace $V\subset  \mathbf{k}^n$ such that $U\cap V=0$. We first count the number of unordered bases $\{ a_1, \cdots, a_n\}$ such that $\langle a_1 \rangle=U$ and $\langle a_2, \cdots , a_n \rangle =V$. There are $q-1$ choices that will give $U$ and using Lemma 
\ref{basis_count} there are $\frac{1}{(n-1)!}\prod_{i=0}^{n-2}(q^{n-1}-q^i)$ choices that will give $V$, and so there are $\frac{q-1}{(n-1)!}\prod_{i=0}^{n-2}(q^{n-1}-q^i)$ total choices. Quotienting out by the free $\mathbf{k}^*$ action gives the result. 

\end{proof}

\begin{lemma}\label{divides}
Suppose there is a map $g\colon A_{\mathbf{k}^n} \rightarrow K_n$, then $n|(q^n-1)q^{n-1}$.  
\end{lemma}

\begin{proof}
Since every graph homomorphism $K_n\rightarrow K_n$ is an isomorphism, the composition $g \circ F\colon \coprod_{B\in \mathcal{B}} K_n \rightarrow K_n$, must be an isomorphism when restricted to each $K_n$ in the coproduct. Hence using 
Lemma \ref{basis_count}
we see that 
for any $x\in V(K_n)$, $|(g\circ F)^{-1}(x)|=|\mathcal{B}|=\frac{1}{n!}\prod_{i=0}^{n-1}(q^{n}-q^i)=\frac{1}{n!}(q^n-1)(q^{n-1})\prod_{i=0}^{n-2}(q^{n-1}-q^i)$. 

Next using Lemma \ref{inverseAn}, 
\begin{align*}
|(g\circ F)^{-1}(x)|=&\Sigma_{(U,V)\in g^{-1}(x)} F^{-1}(U, V)\\
=&\Sigma_{(U,V)\in g^{-1}(x)} \frac{1}{(n-1)!}\prod_{i=0}^{n-2}(q^{n-1}-q^i)\\
=&|g^{-1}(x)|\frac{1}{(n-1)!}\prod_{i=0}^{n-2}(q^{n-1}-q^i).
\end{align*}

So we get that $\frac{1}{(n-1)!}\prod_{i=0}^{n-2}(q^{n-1}-q^i)$ divides $\frac{1}{n!}(q^n-1)(q^{n-1})\prod_{i=0}^{n-2}(q^{n-1}-q^i)$, multiplying both by $n!$ and canceling terms we get the result. 
\end{proof}

\begin{proposition}\label{cong}
   Let $p$ by a prime. If $q\not\equiv 0,1 \ (mod \ p)$ then there are no maps $A_{\mathbf{k}^p}\rightarrow K_p$ and so $\chi(A_{\mathbf{k}^p})>p=s_p\chi(A_{\mathbf{k}^p})$. 
\end{proposition}
\begin{proof}
   From Lemma \ref{divides} if there were a map $A_{\mathbf{k}^p}\rightarrow K_p$ we would have $p|(q^n-1)q^{n-1}$. 
   Thus if $p\not| q$ then $p|(q^n-1)$. 
   Fermat's Theorem says that  $q^n\equiv q \ (mod \ p)$ and so $q\equiv 1 \ (mod \ p)$.
   Now $q\not\equiv 0,1 \ (mod \ p)$, it is contradiction, and there is no map $A_{\mathbf{k}^p}\rightarrow K_p$ which proves the inequality. The equality is just Proposition \ref{numberofAn}.
\end{proof}

By using these properties we obtain the following.
\begin{proposition}\label{chromatic_An}
   For any odd prime $p$, $\chi(A_{(\zZ/2)^p})>p=s_2\chi(A_{(\zZ/2)^p})$.
\end{proposition}
\begin{proof}
   As $2\not\equiv 0,1 \ (mod \ p)$ for any odd prime, we use Proposition \ref{cong}.
   \end{proof}

\begin{example}\label{smallestexample}
 From the proposition $\chi(A_{(\zZ/2)^3})>3$, and from Proposition \ref{numberofAn}
 $s_{\mathbf{2}}\chi(A_{(\zZ/2)^3})=3$. 
\end{example}

\begin{proposition}\label{chromatic_difference}
   For any finite field $\mathbf{k}$, there is a graph $G$ such that
   \[
      s_{\mathbf{k}}\chi(G)<\chi(G).
   \]
\end{proposition}
\begin{proof}
   Since $q$ is no less than $2$, there is a prime number $p$ such that $q\not\equiv 0,1 \ (mod \ p)$. Then for the case $G=A_{\mathbf{k}^p}$, there is an inequation $s_{\mathbf{k}}\chi(G)=p<\chi(G)$.
\end{proof}

There is also a characteristic $0$ example. 

\begin{example}\label{God}
Consider the two sphere $S^2\subset \mathbb{R}^3$. So $S^2=\{ x\in \mathbb{R}^3| |x|=1\}$. 
We can turn $S^2$ into a graph with 

$$V(S^2)=S^2$$
$$ E(S^2)=\{ (x,y)| x\cdot y=0\}$$

So the neighborhood of $x$, $N(x)$ consists of the great circle orthogonal to $x$. 
Define $f\colon S^2 \rightarrow A_{\mathbb{R}^3}$ by $f(x)=(\langle x\rangle, \langle N(x)\rangle)$. 
This gives a 3-span colouring of $S^2$ over $\mathbb{R}$ and thus $s_{\mathbb{R}}\chi(S^2)=3$. 
On the other hand it is observed in \cite{GZ} that $\chi(S^2)=4$. 

Note this colouring lands in the subgraph $B$ of $A_{\mathbb{R}^3}$ consisting of $(U,V)$ with $U$ and $V$ being orthogonal. Is there a retraction  $A_{\mathbb{R}^3}\rightarrow B$ of this inclusion? In other words do both graphs give colourings with the same chromatic number. 
\end{example}

\section{Span colouring and Steenrod algebra}\label{sec:Span_colouring_and_Steenrod_algebra}

In the rest of the paper, we will mostly be concerned with the realization problem for the algebras $A(n,L)$. In this section 
for any graph $G$, we show that Steenrod algebra actions on $A(n,G)$ 
give rise to $n$-span colourings on $G$ (Theorem 
\ref{Steenrod_alg_chromatic_number}). First we recall the Steenrod Algebra.

\subsection{The Steenrod Algebra and its action on Stanley-Reisner rings}

For a prime $p$, let $\SA_p$ denote the mod-$p$ Steenrod algebra. 
Let $A^*$ be a commutative graded algebra with unit over $\zZ/p$.
When $A^*$ has an action of $\SA_p$, we call $A^*$ an algebra over the (mod-$p$) Steenrod algebra.
Unless noted we will assume the prime is $2$. 

For an odd prime number $p$, an algebra over the mod-$p$ Steenrod algebra $A^*$ is {\em unstable} if and only if for all homogeneous elements $x\in A^{2d}$, there is an equation
\[
   \beta^\epsilon \Pe^{k}(x)=
      \begin{cases}
          x^p \quad &(2k=2d,\epsilon=0)\\
          0 \quad &(2k+\epsilon>2d),
      \end{cases}
\]
where $\mathcal{P}^k$ is the Steenrod power operation and $\beta$ is the Bockstein homomorphism.
Similarly we can define an algebra over the mod-$2$ Steenrod algebra $A^*$ is {\em unstable} if and only if for all homogeneous elements $x\in A^{d}$, there is an equation
\[
      \Sq^{k}(x)=\begin{cases}
        x^2 \quad &k=d\\
        0 \quad &k>d,
    \end{cases}
\]
where $\Sq^k$ is the Steenrod square.
The interaction of the multiplication of $A^*$ and the action is described by the Cartan formula:

\[
\Pe^{k}(xy)=\Sigma_{0\leq i\leq k} \Pe^i (x) \cup\Pe^{k-i}(y)
\]

for odd characteristic and, 

\[
\Sq^{k}(xy)=\Sigma_{0\leq i\leq k} \Sq^i (x) \cup\Sq^{k-i}(y)
\]

for the prime $2$, where $\Pe^0=\Sq^0=id$. We assume all our actions of $\SA_p$ are unstable and satisfy the Cartan formula. 
Please see \cite{AW,HJS} for more information.

\subsection{Steenrod action on Stanley-Reisner and $P_{\max}(K)$}

Let $K$ be a simplicial complex and $\phi\colon V(K)\rightarrow 2\zN$. We regard $K$ as a poset and
recall $P_{\max}(K)$ is the subposet of $K$ such that for $\sigma \in K$, $\sigma \in P_{\max}(K)$ if and only if there exist maximal simplices $\sigma_1, \dots \sigma_n \in K$ such that $\bigcap \sigma_i =\sigma$.

The following lemma is equivalent to Lemma 6.2 in \cite{Ta}.
\begin{lemma}\label{quotient_to_sum_of_maximal_cell}
   Let $SR(K,\phi)$ be a graded Stanley-Reisner ring for a simplicial complex $K$ with the vertex set $V$ and $\phi\colon V\rightarrow 2\zZ_{>0}$.
   We assume that for a prime number $p$, $SR(K,\phi)\otimes \zZ/p$ has an $\SA_p$ action. 
   Then for $\sigma \in P_{\max}(K)$, the ideal $(V\setminus \sigma)$ in $SR(K,\phi)\otimes \zZ/p$ preserves the action of the mod-$p$ Steenrod algebra.

    Also for any $U\subset P_{\max}(K)$ , the ideal 
   $
     \bigcap_{\tau \in U}(V\setminus \tau)
   $
   in $SR(K,\phi)\otimes \zZ/p$ preserves the action of the mod-$p$ Steenrod algebra.

\end{lemma}
\begin{proof}
   There are maximal cells $\sigma_1,\sigma_2,\dots \sigma_n$ such that $\sigma=\sigma_1\cap \sigma_2 \cap \dots \cap \sigma_n$.
   For $x_i\in V\setminus \sigma$ there is $1 \leq j\leq n$ such that $x_i\in V\setminus \sigma_j$.
   By Lemma 6.1 in \cite{Ta}, for any element of the Steenrod algebra $\theta$, $\theta(x_i)$ is in $(V\setminus \sigma_i)\subset (V\setminus \sigma)$.
   Thus the generators in $(V\setminus \sigma)$ preserve the action of the mod-$p$ Steenrod algebra, and using the Cartan formula the ideal $(V\setminus \sigma)$ preserves the action showing the first statement. 
   
   Therefore the intersection of them also preserves the action and the second statement is also true. 
\end{proof}

We relate these intersections to the ideals $I_K$.  Recall that a simplex $\sigma\in K$ is also considered as a subcomplex
$\sigma\subset K$. 

\begin{lemma}\label{wipeout}
Let $K$ be a simplicial complex and $U\subset K$ a subset that contains all maximal simplices of $K$, 
in other words such that $\bigcup_{\sigma\in U}\sigma=K$. Then 
$I_K=\bigcap_{\sigma\in U}(V\setminus \sigma)$.
Suppose $U'\subset U$. Let
$K'=\bigcup_{\sigma\in U'}\sigma\subset K$
and $I=\bigcap_{\sigma\in U'}(V\setminus \sigma)\subset SR(K,\phi)$, then 
$I$ is the kernel of the induced map $SR(K,\phi)\rightarrow SR(K'\phi)$. 
\end{lemma}

\begin{proof}
Let $\alpha\in SR(K)$ be a monomial. The support of $\alpha$, $supp(\alpha)\subset V(K)$ is defined as 
the elements that appear with non-zero exponent in $\alpha$. 
The monomial $\alpha$ is in $(V\setminus \sigma)$ if and only if $supp(\alpha)\not\subset \sigma$.  
Also $\alpha\in I_K$ if and only if for every $\sigma\in K$, 
$supp(\alpha)\not\subset \sigma$. So $\alpha\in I_K$ if and only if
for every $\sigma \in K$, $\alpha \in (V\setminus \sigma)$ if and only if
$\alpha\in \bigcap_{\sigma\in K}(V \setminus \sigma)$. 

Since both $I_K$ and $\bigcap_{\sigma\in K}(V \setminus \sigma)=\bigcap_{\sigma\in U}(V\setminus \sigma)$ are determined by the monomials they contain, the first statement of the lemma follows. The proof of the second statement is similar. 
\end{proof}

\begin{example}
Let $\sigma \in K$ be a simplex, then $K_{\sigma}=\sigma$ and the last lemma says that 
$SR(K)/(V\setminus \sigma)\cong SR(\sigma, \phi)=\zZ[\sigma]$.
\end{example}

Combing the last two lemmas we get:

\begin{proposition}\label{quotSt}
Suppose $K$ is a simplicial complex and $U\subset P_{\max}(K)$ and let $K'=\bigcup_{\sigma\in U}\sigma$. If 
$SR(K,\phi)\otimes \zZ/p$ has an $\SA_p$ action then there is a unique $\SA_p$ action 
on $SR(K',\phi)\otimes \zZ/p$ making the quotient map $SR(K,\phi)\otimes \zZ/p\rightarrow SR(K',\phi)\otimes \zZ/p$ a map of algebras over the Steenrod algebra. 

In particular if $\sigma\in P_{\max}(K)$ then $\zZ/p[\sigma]$ has an induced action of $\SA_p$. 
\end{proposition}

\begin{proof}
This follows directly from Lemmas \ref{quotient_to_sum_of_maximal_cell} and \ref{wipeout} with uniqueness following since the quotient map is surjective. 
\end{proof}


\subsection{Steenrod actions on Stanley-Reisner algebras generated in degrees $4$ and $6$}

Now we restrict our attention Stanley-Reisner algebras generated in degrees $4$ and $6$.
Let $SR(K,\phi)=\zZ[x_1,\dots x_n,y_1,\dots y_m]/I_K$ be a Stanley-Reisner ring with $|x_i|=4,|y_i|=6$. By the unstable condition, an unstable mod-$2$ Steenrod algebra action on $SR(K,\phi)$ is determined by $\Sq^2(x_i),\Sq^2(y_i),\Sq^4(y_i)$, because mod-$2$ Steenrod algebra is generated by $\Sq^{2^i}$, and by the unstable condition $\Sq^{2^j}(x_i)=\Sq^{2^j}(y_i)=0$ for $j\geq 3$ and $\Sq^{4}(x_i)={x_i}^2$.
In other words, if there exist two unstable actions on $SR(K,\phi)$ such that $\Sq^2(x_i),\Sq^2(y_i),\Sq^4(y_i)$ are same, then using the Cartan formula we see that the two actions are the same. 
This is just saying that the action is determined by what it does on algebra generators. In the next section to see particular definitions of the action on generators extend to the whole algebra we will take advantage of a limit description of Stanley-Reisner rings.

At first we consider an unstable mod-$2$ Steenrod algebra action on polynomial rings generated by degree $4$ and $6$ generators. 
The following result follows from the result of Sugawara and Toda \cite{ST}.
\begin{theorem}\label{Sugawara_Toda}
   The polynomial ring $\zZ/2[x_1,\dots x_n,y_1,\dots y_m]$ with $|x_i|=4,|y_i|=6$ has an $\SA_2$ action if and only if $n\geq m$.
\end{theorem}
We would like to construct a Steenrod algebra action on polynomial ring generated by degree $4$ and $6$ generators.
\begin{lemma}\label{Steenrod_structure_on_SU(3)}
   Let $\zZ/2[x_1,\dots x_n,y_1,\dots y_m]$ be a polynomial ring  with $|x_i|=4,|y_i|=6$ and $n\geq  m$.
   Then there exists a Steenrod algebra action on this polynomial ring such that
   \[
      \Sq^4(y_i)=y_i x_i,\Sq^2(y_i)=0, \text{ and } \Sq^2(x_i)=
      \begin{cases}
         y_i \quad &(i\leq m)\\
         0 \quad &(m< i),
      \end{cases}
   \]
   and $\zZ/2[x_1,\dots x_n,y_1,\dots y_m]$ becomes an unstable algebra over mod-$2$ Steenrod algebra induced by this action.
\end{lemma}
\begin{proof}
   By Wu formula \cite{Sh}, the cohomology of the classifying space of $SU(3)$, $H^*(BSU(3),\zZ/2)\cong \zZ/2[x,y]$ is the unstable algebra over Steenrod algebra defined by
   \[
      \Sq^4(y)=y x, \,\Sq^2(y)=0, \text{ and } \Sq^2(x)=y.
   \]
   Similarly, the unstable algebra structure over Steenrod algebra on $H^*(BSU(2);\zZ/2)\cong \zZ/2[x]$ is  defined by $\Sq^2(x)=0$.
   Therefore there is an isomorphism as unstable algebra over Steenrod algebras
   \[
      \zZ/2[x_1,\dots x_n,y_1,\dots y_m]\cong \bigotimes_{i=1}^m H^*(BSU(3);\zZ/2)\otimes \bigotimes_{i=1}^{n-m} H^*(BSU(2);\zZ/2).
   \]
\end{proof}

\subsection{When $A(n,G)$ has a Steenrod algebra action}\label{abSteen}

 For the rest of the section we fix a graph $G$ and assume $A(n,G)\otimes \zZ/2$ has an $\SA_2$ action.

\begin{lemma}\label{preserve_pre}
  Assume that the minimum degree of the graph $G$ is no less than $2$, i.e. each vertex in $G$ has at least $2$ adjacent vertices.
   Then for all $i$, $\Sq^4(y_i)$ is in the ideal $(y_i)$ in $A(n,G)\otimes \zZ/2$.
\end{lemma}
\begin{proof}
 Using Lemma \ref{46abthings} (1), all maximal simplices in $K$ contain $\{x_1,\dots x_n\}$ as a subset.
   Next by Lemma \ref{46abthings} (2), there is no triple $1\leq i<j<k\leq m$ such that $y_iy_jy_k\ne 0$.
   Thus for $i,j$ with $y_iy_j\ne 0$ $\sigma=\{x_1,\dots x_n,y_i,y_j\}$ is a maximal simplex, and $(V\setminus \sigma)=(\{y_1,y_2,\dots y_n\}\setminus\{y_i,y_j\})$.

   Let $U$ be the set of all maximal simplices not containing $y_i$.
  Since each vertex of $G$ has degree at least $2$, for any $j\not=i$ there exists $\tau\in U$ with $y_j\in \tau$. 
   So $\bigcup_{\tau\in U}\tau=K_{V(K)-\{y_i\}}$ and using Lemma \ref{wipeout} we obtain the inclusion

   \[
      \bigcap_{\sigma\in U} (V\setminus\sigma)
      =I_{\bigcup_{\tau\in U}\tau}
      \subset (y_i)+ ( y_jy_k\mid 1\leq j < k\leq m).
   \]

   Since by Lemma \ref{quotient_to_sum_of_maximal_cell}, $\Sq^4(y_i)$ is in the ideal 
   \[
      \bigcap_{\sigma\in U} (V\setminus\sigma),
   \]
   we obtain $\Sq^4(y_i)$ is in $(y_i)$ for degree reasons.
\end{proof}
\begin{example}
   We cannot prove this lemma in the general for algebras of the form $A(n,L)$. 
   
      Let $\zZ[x_1,x_2,x_3,y_1,y_2,y_3]$ be a polynomial ring and we define 
   \begin{align*}
      &\Sq^2(x_1)=y_1 &\quad &\Sq^2(x_2)=y_2&\quad &\Sq^2(x_3)=y_1+y_3\\
      &\Sq^4(y_1)=y_1x_1 &\quad &\Sq^4(y_2)=y_2x_2& \quad &\Sq^4(y_3)=y_1x_1+y_1x_3+y_3x_3\\
      &\Sq^2(y_1)=0 &\quad&\Sq^2(y_2)=0&\quad &\Sq^2(y_3)=0.
   \end{align*}
   Then there is an equation $\Sq^4(y_1+y_3)=(y_1+y_3)x_3$.
   When we take $\{x_1,x_2,x_3y_1,y_2,y_1+y_3\}$ as generators of this ring, by Lemma \ref{Steenrod_structure_on_SU(3)} this ring is isomorphic to $H^*(BSU(3))\otimes H^*(BSU(3))\otimes H^*(BSU(3))$.
   Therefore this ring satisfies all condition in the lemma other than this condition, but $\Sq^4(y_3)$ is not in $(y_3)$.

   Similarly, we cannot prove this lemma without the condition: ``the minimum degree of the graph $G$ is no less than
2".
   Let $\zZ[x_1,x_2,y_1,y_2]$ be a polynomial ring and we define 
   \begin{align*}
      &\Sq^2(x_1)=y_1& \quad &\Sq^2(x_2)=y_1+y_2\\
      &\Sq^4(y_1)=y_1x_1& \quad &\Sq^4(y_2)=y_1x_1+y_1x_2+y_2x_2\\
      &\Sq^2(y_1)=0&\quad &\Sq^2(y_2)=0.
   \end{align*}
   Then by a similar argument, we obtain that this ring satisfies all condition in the lemma other than this condition, but $\Sq^4(y_2)$ is not in $(y_2)$.
\end{example}

Recalling the notation of this section and considering degrees, let $J_i\subset [m]$ be such that $\Sq^4(y_i)=y_i(\sum_{j \in J_i} x_j)$.
We define the map $f\colon \{y_1,\dots y_m\}\rightarrow \langle x_1,\dots ,x_n\rangle$ as $f(y_i)=\sum_{j \in J_i} x_j$.
For a degree $6$ generator $y_i$, let $N(y_i)=N_{L_{\phi^{-1}(6)}^1}(y_i)\subset V$ be the neighborhood of $y_i$ in $L_{\phi^{-1}(6)}^1$.
\begin{lemma}\label{neighborhood_independent}
   Assume that $G$ has minimal degree at least $2$, and assume $A(n,G)\otimes \zZ/2$ has an $\SA_2$ action and use the notation above.
    Then for $1\leq i\leq m$ the cokernel of the linear inclusion
    \[
        \left\langle f(y_j) \Bigm\vert y_j \in N(y_i) \right\rangle \rightarrow \left\langle f(y_j) \Bigm\vert y_j \in N(y_i) \cup \{y_i\} \right\rangle
    \]
    is non-trivial.
    Thus $f$ is a $n$-span colouring of $G$ over $\zZ/2$.
\end{lemma}
\begin{proof}
   
   There is an equation $\Sq^2\Sq^4=\Sq^6+\Sq^5\Sq^1$ by the Adem relation. 
   Since there is no odd degree part in $SR(L,\phi)$, by this equation and the definition of $f$ for all $j$ we obtain $\Sq^2(y_jf(y_j))=\Sq^2(\Sq^4(y_j))=y_j^2$. Since we work in a square free monomial ideal ring, 
   by this equation and the Cartan formula, we obtain that  $\Sq^2(f(y_j))$ includes $y_j$ for all $j$.
   More precisely for degree reasons we can write $\Sq^2(f(y_j))=\Sigma_{i\in [m]} c_i y_i$ and we must have that 
   $c_j=1$. 
   
   We assume that the cokernel is trivial for an integer $i$.
   Then there exists an equation 
   \[
      f(y_i)=\sum_{j\in N(y_i)}a_{j}f(y_j)
   \] 
   for $a_{j}\in \zZ/2$.
   Since $\Sq^2(f(y_i))$ includes $y_i$, for some $y_j \in N(y_i)$ $\Sq^2(f(y_j))$ also includes $y_i$.
   Then $\Sq^2(y_j f(y_j))=y_j\Sq^2(f(y_j))+\Sq^2(y_j)f(y_j)$ must contain the term $y_jy_i$ which is not $0$ since 
   $y_i\in N(y_j)$. 
   This contradicts the equation $y_{j}^2=\Sq^2(y_jf(y_j))$.
   Therefore the assumption is wrong, and we obtain that the cokernel is non-trivial.
\end{proof}

Next we will eliminate the condition about the minimal degree of $G$ in Lemma \ref{neighborhood_independent} 
and Lemma \ref{preserve_pre}.
For a graph $G$, there exists a unique maximal subgraph $2G$ of $G$ such that the minimal degree is no less than $2$.
Note that $2G$ can be empty. 
\begin{lemma}
Suppose $G$ is a graph, and 
\begin{equation}\label{happy}
      G=G_0\supset G_1\supset \dots \supset G_l=H
   \end{equation}
   a sequence of subgraphs such that $G_{i+1}$ is gotten from $G_i$ by removing some vertices of degree $0$ and $1$, and such that $H$ as minimal degree no less then $2$. Then $H=2G$. Also such a sequence always exist such that going from $G_{i+1}$ to $G_i$ removes exactly one vertex.
\end{lemma}
\begin{proof}
The graph $H$ is a subgraph of $G$ with minimal degree no less than $2$. Given any other such subgraph $H'$ can not have any of the vertices removed to get to $H$ and so $H'\subset H$ is a subgraph. Thus $H=2G$. It is easy to see that such a sequence exists . 
\end{proof}

Recall the set up from the start of Section \ref{abSteen}.

\begin{lemma}\label{remove_degree_0_1}
   Assume that $2G$ is not the null graph. Then the projection  $A(n,G)\otimes \zZ/2 \rightarrow A(n,2G)\otimes \zZ/2$
induces a Steenrod algebra action on 
  $A(n,2G)\otimes \zZ/2$.
  
  \end{lemma}
\begin{proof}
If $G=2G$ there is nothing to prove. Otherwise there is $y\in V(G)$ such that $deg(y)\leq 1$. 
   When $y$ is a degree $0$ vertex in $G$, then $\sigma_1=\{x_1,\dots x_n,y\}$ is the only maximal simplex of $K$ containing $y$. Since $H$ is not the null complex, it is not the only maximal simplex. Let 
   $\sigma_1,\dots , \sigma _k$ be the maximal simplices of $L$. Directly compute that 
   $K_{V\setminus \{y\}}=\sigma_2\cup \dots \cup \sigma_k$. 
   Also $K_{V\setminus \{y\}}=\Delta^{n-1}\ast G_{V(G)\setminus \{y\}}=A(n,G_{V(G)\setminus \{y\}})$. 
   So $A(n,G_{V(G)\setminus \{y\}})\otimes \zZ/2$ has an induced action of the Steenrod Algebra by Proposition
   \ref{quotSt}

   Similarly when $y$ is a degree $1$ vertex in $L^1$ and $y'$ satisfies 
   $yy'\ne 0$, then $\sigma_1=\{x_1,\dots x_n,y, y'\}$ is the only maximal simplex containing $y$ and let $\sigma_2, \dots ,  \sigma_k$ be the rest of the maximal simplices of 
   $K$. As before $k>1$ since $2G$ is not the null graph. If $y'$  is also of degree one then again by direct computation 
   $\sigma_2\cup \dots \cup \sigma_k=K_{V\setminus \{y, y'\}}$. On the other hand if the degree of $y'$ is larger than $1$, 
    $\sigma_2\cup \dots \cup \sigma_k=K_{V\setminus \{y\}}$.
   
   Arguing as in the degree $0$ case, we get that  $A(n,G_{V(G)\setminus \{y\}})\otimes \zZ/2$ or 
   $A(n,G_{V(G)\setminus \{y,y'\}})\otimes \zZ/2$ have an induced action of the Steenrod Algebra.
   
   If degree $0$ or $1$ vertices remain we can repeat the process and so will get a sequence of the form of Equation 
   \ref{happy} together with the sequence of projections $A(n,G)=A(n,G_0)\rightarrow A(n,G_1) \rightarrow \dots\rightarrow 
   A(n,G_l)=A(n,2G)$ with a Steenrod action on $A(n,G_i)\otimes \zZ/2$ inducing one on 
   $A(n,G_{i+1})\otimes \zZ/2$. Thus the lemma follows. 
\end{proof}

The following is the main result in this section.
The following result is for any $A(n,G)$. 
\begin{theorem}\label{Steenrod_alg_chromatic_number} 
For any graph $G$, 
 if $A(n,G)\otimes \zZ/2$ has an $\SA_2$ action then $s_2\chi(G)\leq n$.
\end{theorem}
\begin{proof}
   Recall $2G$, the maximal subgraph of $G$ such that minimal degree no less than $2$. Then by Lemma \ref{neighborhood_independent} and \ref{remove_degree_0_1}, we obtain that $s_2\chi(2G)\leq n$.

   First suppose $2G$ is not the empty set, then $s_2\chi(2G)\geq 2$.   Take an $n$-span colouring $f\colon 2G\rightarrow A_n$. We will build a $n$-span colouring of $G$ using a sequence \ref{happy}. 
  such that going from $G_{i+1}$ to $G_i$ we add a single degree $0$ or $1$ vertex. Supposing we have 
  $f_{i+1}\colon G_{i+1}\rightarrow A_n$. We set $f_i|_{G_i}=f_{i+1}$. On a degree $0$ vertex, $v$ we can define $f_i(v)$ to be any element of $A_n$. 
  If $v$ has degree $1$, let $w$ be the one vertex it is connected to. We have $f_{i+1}(w)=(U,V)$ were $U$ is a line and since $n\geq 2$, $V$ is a subspace of dimension at least one that does not contain $U$. So pick any line $U'\subset V$ and $W$ a complementary subspace of $U'$ in $V$. Define $f_i(v)=(U', W\oplus U)$. Putting these together we have defined $f_i\colon G_i\rightarrow A_n$. Doing this same procedure repeatedly give us a map
  $G\rightarrow A_n$ and so $s_2\chi(G)\leq n$.   
     
   Next we consider the case that $2G$ is the emptyset. Either by following the argument above or by observing that $G$ is a forest (ie a disjoint union of trees) and applying Proposition \ref{important}
   we see that $s_2\chi(G)$ is at most $2$. Thus if $n\geq 2$ we are done and we only have to deal with the cases $n=0 \ {\rm or } \ 1$. 
   Let $\sigma$ be a maximal simplex of $K$. If $n=0$ then $\sigma=K_\sigma$ is a simplex with no degree $4$ vertices and so by Theorem \ref{Sugawara_Toda} it also has no degree $6$ simplices this implies $G$  is empty so $s_2\chi(G)=0=n$
If $n=1$,  $K_\sigma$ is a simplex with one degree $4$ vertex and so again by  Theorem \ref{Sugawara_Toda}, $\sigma$ has at most one degree $6$ vertex. This implies that $G$ has no edges and so $s_2\chi(G)\leq 1=n$. 
     
 By combining these we can obtain $s_2\chi(G)\leq n$.
\end{proof}

\begin{corollary}\label{realimpspan}
For any graph $G$, if $A(n,G)$ is realizable then $s_2\chi(G)\leq n$.
\end{corollary}

\begin{proof}
Suppose $A(n,G)$ is realized by the space $X$, then 
$A(n,G)\otimes \zZ/2\cong H^*(X, \zZ)\otimes  \zZ/2\cong H^*(X, \zZ/2)$ with the second isomorphism following since $A(n,G)$ is torsion free. Since $H^*(X, \zZ/2)$ has an unstable Steenrod action so does $A(n,G)\otimes \zZ/2$ and then 
Theorem \ref{Steenrod_alg_chromatic_number} implies $s_2\chi(G)\leq n$.
\end{proof}

\section{Construction of Steenrod algebra structure}\label{sec:Construction_of_Steenrod_algebra_structure}

In this section we construct an unstable Steenrod algebra action on $A(n,L)$ 
using an $n$-span colouring on its one 
skeleton $L^1$ (Theorem \ref{construct_Steenrod_mod_str}) which gives a converse to 
Theorem \ref{Steenrod_alg_chromatic_number} when $L=L^1$.  
We first use poset indexed limits to construct Steenrod algebra actions on a 
general graded Stanley-Riesner ring $SR(K, \phi)$.

\subsection{Steenrod actions on limits of polynomial algebras}
We sometimes regard a poset $P$ as a category.
The objects of this category are the elements in $P$, and for $a,b\in P$ there is at most one morphism $a\rightarrow b$ and there exist a morphism $a\rightarrow b$ if and only if $a\leq b$.
Let $F$ be a functor from $P^{op}$ to the category of unstable algebras over mod-$p$ Steenrod algebra.
We define $A$ as 
\[
   A=\left\{(x_a)\in \prod_{a\in P} F(a)\mid F(a>b)(x_a)=x_{b}\right\}.
\]
This $A$ is the algebra over $\zZ/p$ induced by the algebra structure of each $F(a)$ for $a\in P$.

We define the Steenrod algebra action on $x=(x_a)\in A$ as 
\[
   \mathcal{P}^i(x)=(\mathcal{P}^i(x_{a})),\quad \beta(x)=(\beta(x_a))
\]
when $p\geq 3$, and 
\[
   \Sq^i(x)=(\Sq^i(x_{a}))
\]
when $p=2$.

The action of the Steenrod algebra on $A$ as an module is induced by these maps.
Then the projection $A\rightarrow F(a)$ becomes a homomorphism of algebras over mod-$p$ Steenrod algebra.
Thus by considering the each component, we can see that this action satisfies the Cartan formula and the unstable condition.
Therefore $A$ becomes an unstable algebra over mod-$p$ Steenrod algebra.
Regarding this $A$ we have the following lemma.
\begin{lemma}\label{limit_alg_Steenrod_geberal}
   In the situation above, $A$ is the limit of $F$.
   Moreover $A$ is also the limit as a graded algebra over $\zZ/p$. 
   In other words, the forgetful functor from the category of unstable algebras over mod-$p$ Steenrod algebra to graded algebras over $\zZ/p$ preserves the limit.
\end{lemma}
\begin{proof}
   Let $B$ be an algebras over mod-$p$ Steenrod algebra with maps $f_a\colon B\rightarrow F(a)$ for all $a \in P$ that makes the following diagram commutative for all $a <b \in P$
   \[
      \xymatrix{
         B\ar^{f_b}[r] \ar_{f_{a}}[rd]& F(b)\ar^{F(b>a)}[d]\\
          & F(a). \\
      }
   \]
   We define the map $f\colon B\rightarrow A$ as $f(y)=(f_a(y))$.
   By the commutative diagram, the image of $f$ is in $A$ and this map is well-defined.
   Since for $y,y'\in B$ there are equations
   \begin{align*}
      f(yy')=(f_a(yy'))=(f_a(y)f_a(y'))=f(y)f(y')\\
      f(\theta(y))=(f_a(\theta(y)))=\theta(f_a(y))=\theta f(y),
   \end{align*}
   where $\theta$ is a Steenrod operation, the map $f$ is a homomorphism of algebras over mod-$p$ Steenrod algebra.
   Therefore there exists a commutative diagram in the category of algebras over mod-$p$ Steenrod algebra
   \[
      \xymatrix{
         B\ar^{f}[r] \ar_{f_a}[rd]& A\ar[d]\\
          & F(a). \\
      }
   \]
   Moreover if a map $g\colon B\rightarrow A$ makes the diagram above commutative instead of $f$, then $g$ must satisfy $g(y)=(f_{a}(y))$.
   Thus $g=f$ and we obtain the uniqueness of such a map $f$.
   Therefore $A$ is the limit of the functor $F$.

   When we regard $A$ as an algebra over $\zZ/p$, we can prove that $A$ is the limit of the composition of $F$ and the forgetful functor similarly.
\end{proof}
By this lemma, we describe limits in the category of unstable algebras over mod-$p$ Steenrod algebra.
Next we construct Stanley-Reisner rings as a limit.
The following lemma is well known in category theory as the property of initial functor (ref. Chapter IX, Section 3 in \cite{M}).
\begin{lemma}\label{final_functor_limit}
   Let $\mathcal{C}$ be a category, $P$ be a poset, and $F\colon P^{op} \rightarrow \mathcal{C}$ be a functor.
   If a subposet of $P$, $P'$, satisfies the following condition;
   \begin{enumerate}
      \item For $a\in P$ there exists $b\in P'$ with $a\leq b$
      \item For $a\in P$ and a pair $b,b'\in P'$ with $a\leq b,b'$, there exists a sequence in $P'$, $b_1,b_2,\dots b_i$ such that $a\leq b_1,b_2,\dots b_i$ and $b\geq b_1\leq b_2 \geq \dots \geq b_i \leq b'$.
   \end{enumerate}
   Then there is an isomorphism
   \[
      \lim_{P^{op}} F \rightarrow \lim_{(P')^{op}}F,
   \]
   where $\lim_{(P')^{op}}F$ is the limit of $F$ restricted to $(P')^{op}$.
\end{lemma}
The condition in this lemma means that the inclusion $P'\rightarrow P$ is final.
Therefore the limit is preserved by the inclusion between opposite category of them.

Let $K$ be a simplicial complex with the vertex set $V$ and $\phi\colon V\rightarrow 2\zZ_{\geq 1}$. 
$K$ and $P_{\max}(K)$ are posets, so we also regard these as categories.
We define a functor from the opposite category of $K$, $K^{op}$, to the category of graded commutative algebras over $\zZ/p$ for a prime number $p$ as follows.
\begin{enumerate}
   \item For $\sigma\in K$, $F(\sigma)$ is the graded polynomial ring generated by vertices in $\sigma$, $\zZ/p[\sigma]$, with $|x|=\phi(x)$.
   \item For $\sigma <\tau$ in $P$, the morphism $F(\sigma<\tau)$ is the natural projection $\zZ/p[\tau]\rightarrow \zZ/p[\sigma]$.
\end{enumerate}
Then the following lemma holds.
\begin{lemma}\label{limit_alg_Stanley_Reisner_pre}
   About the functor above, there is an isomorphism
   \[
      \lim_{K^{op}}F\cong SR(K,\phi)\otimes \zZ/p.
   \]
\end{lemma}
\begin{proof}
   We prove this lemma by induction on the number of simplices in $K$.
   Let $\sigma$ be a maximal simplex in $K$.
   Let $\partial \sigma$ be the boundary of $\sigma$.
   Then by the assumption of the induction, $\lim_{(K \setminus \{\sigma\})^{op}}F \cong SR(K\setminus \{\sigma\},\phi)\otimes \zZ/p$, $\lim_{(K_{\leq \sigma})^{op}}F \cong \zZ/p[\sigma]$ and $\lim_{(K_{<\sigma})^{op}}F \cong \left( SR(K\setminus\{\sigma\},\phi)\otimes \zZ/p \right)/(V\setminus \sigma)\cong SR(\partial \sigma,\phi)\otimes \zZ/p$.
   It is easy to see that the following diagrams are pullback diagrams
   \[
      \xymatrix{
         SR(K,\phi)\otimes \zZ/p \ar[r] \ar[d] & SR(K\setminus \{\sigma\},\phi)\otimes \zZ/p \ar[d] & & \lim_{K^{op}}F \ar[r] \ar[d]& \lim_{(K \setminus \{\sigma\})^{op}}F \ar[d] \\
         \zZ/p[\sigma] \ar[r]& SR(\partial \sigma,\phi)\otimes \zZ/p, & & \lim_{(K_{\leq \sigma})^{op}}F \ar[r]& \lim_{(K_{< \sigma})^{op}}F,
      }
   \]
   and the isomorphisms induces the maps between these square. So we obtain the isomorphism
   \[
      \lim_{P^{op}}F\cong SR(K,\phi)\otimes \zZ/p.
   \]
\end{proof}
By using Lemma \ref{final_functor_limit}, we can prove that the functor $F$ restricted to $P_{\max}(K)$ has same limit.
\begin{lemma}\label{limit_alg_Stanley_Reisner}
   About the functor above, there is an isomorphism
   \[
      \lim_{P_{\max}(K)^{op}}F\cong SR(K,\phi)\otimes \zZ/p.
   \]
\end{lemma}
\begin{proof}
   Since $P_{\max}(K)\subset K$ contains all maximal cells, for $\sigma\in K$ there exists $\tau \in P_{\max}(K)$ with $\tau \leq \sigma$.
   Since $P_{\max}(K)\subset K$ is closed under intersection, for $\sigma \in K$ and $\tau,\tau' \in P_{\max}(K)$ with $\sigma \leq \tau,\tau'$, the intersection of $\tau$ and $\tau'$, $\tau\cap \tau'$, satisfies $\sigma \leq \tau\cap \tau'\leq \tau,\tau'$ and $\tau\cap \tau'\in  P_{\max}(K)$. 
   Therefore $P_{\max}(K)\subset K$ satisfies the condition in Lemma \ref{final_functor_limit}, and by combining Lemma \ref{final_functor_limit} and Lemma \ref{limit_alg_Stanley_Reisner_pre} we can obtain this lemma. 
\end{proof}

\begin{proposition}\label{limit_Steenrod_Stanley_Reisner}
   Let $K$ be a simplicial complex with vertex set $V$ and $\phi\colon V\rightarrow \zZ_{\geq 1}$.
Suppose that:
   \begin{enumerate}
      \item For all $\sigma \in P_{\max}(K)$, $\zZ/p[\sigma]$ has an unstable algebra structure over mod-$p$ Steenrod algebra.
      \item For $\tau, \sigma \in P_{\max}(K)$ with $\tau \leq \sigma$, the projection $\zZ/p[\sigma]\rightarrow \zZ/p[\tau]$ is a homomorphism of unstable algebra over the Steenrod algebra
   \end{enumerate}
   Then $SR(K,\phi)\otimes \zZ/p$ has an unstable algebra structure over mod-$p$ Steenrod algebra.
\end{proposition}
\begin{proof}
   By the condition in this statement, we can define a functor $F$ from $P_{\max}(K)^{op}$ to the category of unstable algebras over the Steenrod algebra as $F(\sigma)\cong \zZ/p[\sigma]$ and $F(\sigma<\tau)$ the projection for $\sigma<\tau$ in $P_{\max}(K)$.
   By Lemma \ref{limit_alg_Steenrod_geberal} there exists a limit of $F$ and the limit is preserved by the forgetful functor to the category of algebra over $\zZ/p$.
   By Lemma \ref{limit_alg_Stanley_Reisner} the limit of $F$ is isomorphic to $SR(K,\phi)$ as algebra over $\zZ/p$.
   Therefore we obtain that $SR(K,\phi)$ has an unstable algebra structure over mod-$p$ Steenrod algebra.
\end{proof}

\subsection{Steenrod actions from span colourings}

\begin{lemma}\label{construct_s}
Suppose $G$ is a graph with vertex set $\{ y_1,\dots y_m \}$, and let 
$f\colon \{ y_1,\dots y_m \} \rightarrow \zZ/2\langle x_1,\dots x_n \rangle$ be an $n$-span colouring. 
 Then there are maps 
   $s_{i}\circ p\colon \zZ/2\langle x_1,\dots x_n \rangle \rightarrow \zZ/2\langle y_i \rangle$ such that $s_i\circ f(y_i)=y_i$ and for all $y_j\in N(y_i)$, $s_i\circ f(y_j)=0$.
\end{lemma}

\begin{proof}
Consider the exact sequence 
   \[
   0\rightarrow \zZ/2\langle  f( N(y_i))\rangle \rightarrow \zZ/2\langle x_1,\dots x_n \rangle 
   \stackrel{p}{\rightarrow} M\rightarrow 0
   \]
   By the definition of span colouring $p\circ f(y_i)\in M$ is not $0$ 
   so there is a linear map $\overline{s_i}\colon M\rightarrow \zZ/2\langle y_i \rangle$ such that 
   $ \overline{s_i}\circ p\circ f(y_i)=y_i$. So letting 
    $s_{i}=\overline{s_i}\circ p\colon \zZ/2\langle x_1,\dots x_n \rangle \rightarrow \zZ/2\langle y_i \rangle$ we get that $s_i\circ f(y_i)=y_i$ and for all $y_j\in N(y_i)$, $s_i\circ f(y_j)=0$.
\end{proof}

\begin{theorem}\label{construct_Steenrod_mod_str}
Suppose $L$ is a simplicial complex such that 
    $s_2\chi(L^1)\leq n$. 
   Then $A(n,L) \otimes \zZ/2$ has an $\SA_2$ action. 
\end{theorem}

\begin{proof}
   By assumption there is a weak span colouring of $L^1$, $f\colon V(L^1)=\{ y_1,\dots y_m \} \rightarrow \zZ/2\langle x_1,\dots x_n \rangle$.

   Let $s\colon \zZ/2\langle x_1,\dots x_n \rangle \rightarrow \zZ/2\langle y_1,\dots y_m \rangle$ be defined by $s(x)=\sum_{i} s_i(x)$, where the $s_i$ come from Lemma \ref{construct_s}
   Now give $SR(K,\phi)$ the structure of an unstable algebra over mod-$2$ Steenrod algebra by defining 
  \begin{equation}\label{steen!}
  \Sq^4(y_i)=y_i f(y_i),\Sq^2(y_i)=0, \ {\rm and}\  \Sq^2(x_i)=s(x_i).
  \end{equation}

   As observed after Lemma \ref{quotient_to_sum_of_maximal_cell} these equations would determine an action if the equations are compatible with relations, we show this indirectly using our limit description of $SR(K,\phi)$.
   
   Let $\sigma=\{x_1,\dots ,x_n,y_{j_1},\dots,y_{j_l}\} \in P_{\max}(K)$, and 
   $\pi\colon \zZ/2\langle y_1,\dots y_m\rangle\rightarrow \zZ/2\langle y_{j_1},\dots ,y_{j_l} \rangle$ 
   be the projection. We define an action on 
    $\zZ/2[\sigma]$ using formulas \ref{steen!} which can easily seen to be well defined on the projection. 
   Since the $y_{j_1},\dots,y_{j_l}$ are all in neighborhoods of each other, $\pi\circ s\circ f$ is identity on $\zZ/2\langle y_{j_1},\dots,y_{j_l} \rangle$ and let
   $\{x'_1,\dots x'_{n-l}\}$ be a basis of the kernel of the composition of $\pi\circ s$.
   
   Then there is a isomorphism
   \[
      \zZ/2[\sigma]\cong \zZ/2[f(y_{i_1}),\dots , f(y_{j_l}), x'_1,\dots x'_{n-l}, y_{j_1},\dots,y_{j_l}],
   \]
   and the Steenrod algebra action defined above induces the action same as in Proposition \ref{Steenrod_structure_on_SU(3)}.
   Therefore the formulas \ref{steen!} give $\zZ/2[\sigma]$ an unstable algebra structure over mod-$2$ Steenrod algebra. We just need to observe that these formulas are compatible with maps in our limit diagram. 
   
   For $\sigma>\tau \in P_{\max}(K)$, the projection $\pi\colon \zZ/2[\sigma]\rightarrow \zZ/2[\tau]$ is an algebra map, so we just need to see it commutes with the Steenrod algebra action. Since $ \zZ/2[\sigma]$ has a basis of monomials it is enough to show for any $k$ and any monomial 
   $\alpha=x_1^{\alpha_1} \dots x_n^{\alpha_n}y_1^{\beta_1}\dots y_m^{\beta_m}$, 
   $\Sq^k(\pi(\alpha))=\pi\Sq^k(\alpha)$. To see this expand $\Sq^k(\alpha)$ using the Cartan formula. The result will be a sum of products of terms of the form $\Sq^l(y_i)$ with $l=0,2,4$ and $\Sq^l(x_i)$ with $l=0,2$. For such terms we can directly check using \ref{steen!} that $\pi\Sq^l(y_i)=\Sq^l(\pi(y_i))$ and $\pi\Sq^l(x_i)=\Sq^l(\pi(x_i))$. Also using the Cartan formulas with the other side gives the desired identity, and shows the projection is a map of unstable algebras over mod-$2$ Steenrod algebra. 
   
   Thus this case satisfies the condition of Lemma \ref{limit_Steenrod_Stanley_Reisner}, and we obtain that $SR(K,\phi)$ has an $\SA_2$ action by Lemma \ref{limit_Steenrod_Stanley_Reisner}.
 \end{proof}
 
 When the simplicial complex is a graph the last theorem is the converse to Theorem \ref{Steenrod_alg_chromatic_number}, so we get a theorem from the introduction. 
 
\begin{theorem}\label{sameSteen}
For any graph $G$, 
$A(n,G)\otimes \zZ/2$ has an $\SA_2$ action if and only if $s_2\chi(G)\leq n$. 
\end{theorem}

\begin{proof}
Theorems  \ref{Steenrod_alg_chromatic_number} and \ref{construct_Steenrod_mod_str} 
provide the two directions. 
\end{proof}

In a similar way to Theorem \ref{construct_Steenrod_mod_str} we can prove the following theorem.
\begin{theorem}\label{construct_Steenrod_mod_str_p_ver}
   Let $p$ be a prime number and with $p\equiv 5 \mod 6$ and $L$ be a simplicial complex such 
   that $s_p\chi(L^1)\leq n$.
   Then $A(n,L) \otimes \zZ/p$ has an $\SA_p$ action. 
\end{theorem}

The statement in this theorem is trivial when $p=3$ or $p\equiv 1 \mod 6$, because in these cases the polynomial ring generated by degree $6$ elements has an unstable algebra structure over mod-$p$ Steenrod algebra (cf. \cite{CE}).
To prove this theorem, we use the following lemma instead of Lemma \ref{Steenrod_structure_on_SU(3)}.
\begin{lemma}\label{Steenrod_structure_on_SU(3)_modp}
   Let $\zZ/p[x_1,\dots x_n,y_1,\dots y_m]$ be a polynomial ring  with $|x_i|=4,|y_i|=6$ and $n\geq m$.
   We define a Steenrod algebra action on this polynomial ring as follows
   \begin{align*}
      \Pe^1(y_i)&=2\sum_{2i_2+3i_3=p+2}(-1)^{i_2+i_3+1}\frac{(i_2+i_3-1)!}{i_2!i_3!}x_i^{i_2}y_i^{i_3}\\
      \Pe^1(x_i)&=
      \begin{cases}
         \sum_{2i_2+3i_3=p+1}(-1)^{i_2+i_3+1}\frac{(i_2+i_3-1)!}{i_2!i_3!}x_i^{i_2}y_i^{i_3} \quad &(i\leq m)\\
         0 \quad &(m< i),
      \end{cases}
   \end{align*}      
   then $\zZ/p[x_1,\dots x_n,y_1,\dots y_m]$ becomes an unstable algebra over mod-$p$ Steenrod algebra induced by this action.
\end{lemma}
The definition of the action corresponds to the action on $H^*(BSU(3))$ which is computed by the mod-$p$ Wu formula \cite{S} or Girard's formula.
By using this lemma and Proposition \ref{limit_Steenrod_Stanley_Reisner}, we can prove Theorem \ref{construct_Steenrod_mod_str_p_ver} by similar way to the proof of Theorem \ref{construct_Steenrod_mod_str}.

Since $\zZ/3[x]$ with $|x|=4$ and $\zZ/3[y]$ with $|y|=6$ have $\SA_3$ actions so the converse of 
Theorem  \ref{construct_Steenrod_mod_str_p_ver} corresponding to Theorem \ref{Steenrod_alg_chromatic_number} fails wildly. In fact for any $SR(K,\phi)$ generated in degrees $4$ and $6$, $SR(K,\phi)\otimes \zZ/3$ has an $\SA_3$ action.

\section{Comparing combinatorial and topological invariants}\label{sec:Proof_of_the_main_theorem}

In this section we introduce a topological chromatic number $\chi_{Top}$ and show for any graph $G$
$s_2\chi(G)\leq \chi_{Top}(G)\leq \chi(G)$. We also classify the realizable Stanley-Reisner algebras with two degree
$4$ generators and any number of degree $6$ generators. 

The following is Theorem $1.3$ in \cite{Ta} applied to a Stanley-Reisner ring generated by degree $4$ and $6$ elements.

\begin{theorem}\label{previous_theorem}
   Let $SR(K,\phi)\cong \zZ[x_1,\dots x_n,y_1,\dots y_m]/I$ be a graded Stanley-Reisner ring for a simplicial complex $K$ with the vertex set $V$ and $\phi \colon V\rightarrow 2\zZ_{>0}$ with $\phi(x_i)=4,\phi(y_i)=6$.
   If $SR(K,\phi)$ satisfies the following condition, we can construct a space $X$ such that $H^*(X;\zZ)\cong SR(K,\phi)$. 
   \begin{enumerate}
       \item There is a decomposition of $V$, $\coprod_{i}A_{i} = V$, such that for all $i$ and $\sigma\in P_{\max}(K)$ the multiset $\{\phi(x)\mid x\in \sigma \cap A_i \}$ is equal to $\{4,6\}, \{4\}$ or $\emptyset$.
   \end{enumerate}
\end{theorem}

\begin{example}
Consider $\zZ[x_1,y_1,y_2]$ then $P_{\max}(L)$ consists of a single simplex $\sigma=(x_1, y_1, y_2)$ and if $A_1=\{ x_1, y_1, y_2\}$ then $\phi(x_1)\in \sigma \cap (A_1)=\{4,6,6\}$.  We can not apply the theorem and indeed this algebra is not realizable and can not even have an action of the unstable Steenrod algebra by Theorem \ref{Sugawara_Toda}.
\end{example}

When restricting to the algebras $A(n,L)$ the condition in the theorem can be expressed in terms of the colouring condition given in the next purely combinatorial lemma. 

\begin{lemma}\label{sticky}

For any graph $G=(V,E)$, 
The following are equivalent:
\begin{enumerate}

\item $\chi(G)\leq n$

\item for any simplicial complex $K=\Delta^{n-1}\ast L$ with $L^1=G$, note that $V(K)=V\cup V(\Delta^{n-1})$, and there is a decomposition $\coprod_{i}A_{i} = V(K)$, such that for all $i$ and $\sigma\in P_{\max}(L)$, $0\leq |(\sigma\cap A_i)\cap V(L)|\leq |(\sigma\cap A_i)\cap V(\Delta^{n-1})|=1$

\item
Considering the map $\phi\colon V(K)\rightarrow \{4, 6\}$ such that  $V(\Delta^{n-1})=\phi^{-1}(4)$   and 
$V(L)=\phi^{-1}(6)$.
For any simplicial complex $K=\Delta^{n-1}\ast L$ with $L^1=G$, there is a decomposition of $V(K)$, $\coprod_{i}A_{i} = V$, such that for all $i$ and $\sigma\in P_{\max}(L)$ the multiset $\{\phi(x)\mid x\in \sigma \cap A_i \}$ is equal to $\{4,6\}$ or $\{4\}$.

\end{enumerate}

\end{lemma}

\begin{proof}

It is straightforward to check that conditions $(2)$ and $(3)$ are equivalent. 

Assume $\chi(G)\leq n$, we will prove $(2)$. Basically the colouring is the same information as the decomposition. 
Let $f\colon L^1\rightarrow \{1, \dots n\}$, $V(\Delta^{n-1})=\{ x_1, \dots, x_n\}$ be a colouring 
and $A_i=f^{-1}(i)\cup \{ x_i\}$. 
Let $\sigma\in P_{\max}(L)$ note that each $x_i \in \sigma$ as it is in each maximal simplex of $K$. 
So $\sigma \cap A_i$ consists of $x_i$ together with some vertices of $L$. If it contains more than one vertex, then there must be an edge between them even though they are coloured the same. This means there is at most one vertex of $L$ in $\sigma\cap A_i$. This proves $(2)$. 

Assuming $(2)$, define $f\colon L^1\rightarrow  \{1, \dots n\}$ so that $A_i=f^{-1}(i)\cup \{ x_i\}$. To see this is a colouring note that if there is an edge between two vertices with the same colour then there would be a maximal simplex 
$\sigma$ containing both of those vertices, so $|(\sigma\cap A_i)\cap V(L)|>1$ giving a contradiction. 
\end{proof}

So we get a realizabilty result for the $A(n,L)$ based on a colouring of  $L^1$.

\begin{theorem}\label{color_theorem}
   Suppose $L$ is a simplicial complex.
If $\chi(L^1)\leq n$ then   $A(n,L)$ is realizable. In other words we can construct a space $X$ such that $H^*(X;\zZ)\cong A(n,L)$. 
\end{theorem}

\begin{proof}
Since $(3)$ and $(1)$ are equivalent in Lemma \ref{sticky} and $(3)$ is stronger than the decomposition hypothesis of Theorem \ref{previous_theorem}, 
$\chi(L^1)\leq n$ will imply $A(n,L)$ is realizable using
Theorem 
\ref{previous_theorem} . 
\end{proof}

\subsection{A topologically defined chromatic number}

\begin{definition}
Let $G$ be a graph. Define the {\em topological chromatic number} of $G$, $\chi_{Top}(G)$ to be the least $n$ such that $A(n,G)$ is realizable. 
\end{definition}

When the simplicial complex $L$ is a graph Theorem \ref{color_theorem} together with Corollary \ref{realimpspan} are equivalent to:

\begin{theorem}\label{topboundsfromgraphs}
For any graph $G$, $s_2\chi(G)\leq \chi_{Top}(G)\leq \chi(G)$. 
\end{theorem}

\begin{proof}
 The first inequality is Corollary \ref{realimpspan}
and the second inequality follows from Theorem \ref{color_theorem}. 
\end{proof}

\begin{example}
   We consider the case  
   \[
    A(3,A_{(\zZ/2)^3})=SR(K,\phi)= \zZ[x_1, x_2, x_3,y_1,\dots, y_{28}]/I_K
   \]
 where $A_{(\zZ/2)^3}$ is the graph that we defined in Section \ref{sec:span_coloring}.
   Then by Theorem \ref{construct_Steenrod_mod_str} and $s_2\chi(A_{(\zZ/2)^3})=3$, $SR(L,\phi)\otimes \zZ/2$ has an unstable Steenrod algebra action.
   On the other hand, by Proposition \ref{chromatic_An}, $\chi(A_{(\zZ/2)^3})>3$.
   Therefore we cannot use Theorem \ref{previous_theorem}, and we do not know if $A(3,A_{(\zZ/2)^3})$ is realizable.
\end{example}

\subsection{Stanley-Reisner algebras with two degree $4$ generators and any number of degree $6$ generators}

By Proposition \ref{2_dim_chromatic}, $\chi (G)\leq 2$ if and only if $s_2\chi (G)\leq 2$ for a graph $G$, and we can obtain a necessary and sufficient condition in this case.
\begin{theorem}\label{n_2_if_and_only_if_condition}
   Suppose $SR(K,\phi)=\zZ[x_1,x_2,y_1,\dots y_m]/I_K$ with $\phi(x_i)=4$ and $\phi(y_i)=6$ . 
   Then there is a space $X$ such that $H^*(X;\zZ)\cong SR(K,\phi)$ if and only if $SR(K,\phi)$ satisfies the following conditions.
   \begin{enumerate}
      \item $\chi (K_{\phi^{-1}(6)}^1)\leq 2$
      \item If $x_1y_i=0$ (resp. $x_2y_i=0$) then $x_2y_i\ne 0$ (resp. $x_1y_i\ne 0$).
      \item If $y_iy_j\ne0$, then $x_1x_2y_iy_j\ne 0$.
   \end{enumerate}
\end{theorem}
In this theorem we don't assume the ideal $I_K$ is generated by the products of degree $6$ generators.
If we did assume this we would have an algebra $A(n,L)$ and (2) and (3) would automatically be satisfied. 
\begin{proof}[Proof of Theorem]
   At first we consider the direction that these conditions induces the realizability of $SR(K,\phi)=\zZ[x_1,x_2,y_1,\dots y_m]/I_K$.
   By the first condition, there is a graph homomorphism $f\colon K_{\phi^{-1}(6)}^1\rightarrow K_2$, where $K_2$ is the complete graph on $2$ vertices.
   We regard the vertices of $K_2$ as $x_1,x_2$, the degree $4$ generators of $SR(K,\phi)$. 
   If $y_j$ satisfies $x_1y_j=0$ (resp. $x_2y_i=0$), then $y_i$ doesn't have an adjacent vertex in $K_{\phi^{-1}(6)}^1$, and by the second condition $x_2y_i\ne 0$ (resp. $x_1y_i\ne 0$).
   Thus by changing the image of $f$ on vertices $y_j$ with no adjacent vertex in $K_{\phi^{-1}(6)}^1$ and $f(y_j)y_j=0$, the colouring $f$ can be replaced to satisfy that $y_jf(y_j)\ne 0$.
   We decompose the vertex set $V$ as two sets $A_1=\{x_1\}\cup f^{-1}(x_1)$ and $A_2=\{x_2\}\cup f^{-1}(x_2)$.
   We will check that this decomposition satisfies the condition in Theorem \ref{previous_theorem}.

   Since $\chi (K_{\phi^{-1}(6)}^1)\leq 2$, there is no $3$-cycle in $K_{\phi^{-1}(6)}^1$.
   Therefore the dimension of $K_{\phi^{-1}(6)}$ is at most $1$, and the number of degree $6$ elements in a maximal simplex is at most $2$.
   When $\sigma \in P_{\max}(K)$ contains two degree $6$ generators $y_i,y_j$, by the third condition, $\sigma=\{x_1,x_2,y_i,y_j\}$.
   Since $f$ is a colouring, $f(y_i)\ne f(y_j)$, and the multiset of degrees of generators in $\sigma\cap A_1$ and $ \sigma\cap A_2$ are $\{4,6\}$.
   Thus $\sigma$ satisfies the condition in Theorem \ref{previous_theorem}.
   When $\sigma \in P_{\max}(K)$ contains only one degree $6$ generators $y_j$, $\sigma=\{x_1,y_j\}, \{x_2,y_j\}$ or $\{x_1,x_2,y_j\}$.
   If $\sigma=\{x_1,x_2,y_j\}$, the multiset of degrees of generators in $\sigma\cap A_1$ and $ \sigma\cap A_2$ are $\{4,6\}$ or $\{4\}$ and it satisfies the condition in Theorem \ref{previous_theorem}.
   If $\sigma=\{x_1,y_j\}\in P_{\max}(K)$, there exists a maximal simplex that contains $x_1,y_j$ and doesn't contain $x_2$.
   Since by the third condition such a maximal simplex doesn't contains degree $6$ generator other than $y_j$, the maximal simplex is $\sigma$ itself.
   Therefore $x_2y_j=0$ and since we take $f$ as $f(y_j)y_j\ne 0$ for all $j$, the multiset of degrees of generators in $\sigma\cap A_1=\{4,6\}$ and $ \sigma\cap A_2=\emptyset$.
   By the same argument, in the case $\sigma=\{x_2,y_j\}$ we can obtain that $\sigma\cap A_1=\emptyset$ and $\sigma\cap A_2=\{4,6\}$.
   When $\sigma \in P_{\max}(K)$ doesn't contain a degree $6$ generator, this case satisfies the condition in Theorem \ref{previous_theorem} trivially.
   Therefore for all $\sigma\in P_{\max}(K)$ satisfies the condition in Theorem \ref{previous_theorem} and we obtain that this $SR(K,\phi)$ is realizable.   
   
   Next we consider the other direction.
   We assume that $SR(K,\phi)$ is realizable, and we prove that $SR(K,\phi)$ satisfies the conditions in the statement.
   Now $SR(K,\phi)\otimes \zZ/2$ has an $\SA_2$ action. 
   We consider a generator $y_j$ with $x_2y_j=0$ or $x_1y_j=0$.
   Let $\sigma$ be a maximal simplex that contains $y_j$ then
   $\sigma $ contains at most $1$ degree 4 generator.
   Then by Proposition \ref{quotSt}, $\zZ/2[\sigma]$ has an $\SA_2$ action, and by Theorem \ref{Sugawara_Toda} the multiset of degrees of generators in $\sigma$ must be $\{4,6\}$.
   This means $x_2y_j\ne 0$ or $x_1y_j\ne 0$, and we obtain the second condition.
   Similarly when $y_iy_j\ne 0$, the maximal simplex that contains $y_i,y_j$ must contains $x_1,x_2$.
   Therefore we can obtain the third condition.
   
   It remains to show that the first condition.
   Since $SR(K,\phi)$ is realizable  $SR(K,\phi)\otimes \zZ/2$ has an action of $\SA_2$. 
   We consider two cases. 
   Suppose $x_1x_2=0$ and let $\sigma$ be a maximal simplex of $K$. So $\sigma$ as at most one of 
   $x_1$ and $x_2$. By Proposition \ref{quotSt} 
   $\zZ[\sigma]$    
 has an action of $\SA_2$. So by \ref{Sugawara_Toda}  , $\phi(\sigma)=\{4\} \text{ or } \{4,6\}$, meaning $\sigma$ has at most one degree $6$ vertex and so the graph $K^{1}_{\phi^{-1}(6)}$ is discrete and $\chi(K^{1}_{\phi^{-1}(6)})=1<2$. 
 
 Next consider the case $x_1x_2\not=0$. Let $y_i$ be such that $x_1y_i=0$ or $x_2y_i=0$.  Suppose $\sigma$ is maximal and contains $y_i$, then $\sigma$ doesn't have one of $x_1$ or $x_2$. So arguing as above $y_i$ is not in any of the edges of $K^1_{\phi^{-1}(6)}$. 
 Also $\bigcup_{\tau\in P_{max(K)}\setminus \sigma} \tau= K_{V(K)\setminus \{y_i\}}$ so by Proposition \ref{quotSt} 
 $SR(K_{V(K)\setminus \{y_i\}}, \phi)$ has an $\SA_2$ action. 
 We can continue to remove vertices to get $S=\{ y_{i(1)}, \dots , y_{i(k)}\}\subset \phi^{-1}(6)$ such that 
 for every $y_j\not\in S$, $x_1y_j\not=0$ and $x_2y_j\not=0$ and
 $K_{\phi^{-1}(6)}=K_{\phi^{-1}(6)-S}\coprod S$ as simplicial complexes. 
 
 Then $K'=K_{V(K)-S}$ is of the form $\Delta^1\ast L$
  and $SR(K', \phi)=A(n,L)$ has an $\SA_2$ action. For $\sigma\in K'$ maximal 
 again by Proposition \ref{quotSt} we get an induced $\SA_2$ action on $\zZ/2[\sigma]$, so again by 
  \ref{Sugawara_Toda} 
  since $\sigma$ has at most two degree $4$ generators it has at most two degree $6$ generators. 
 Thus $K'_{\phi^{-1}(6)}$ has no 2-simplices and so $SR(K', \phi)$ is of the form $A(n,G)$. Thus 
 $\chi(K'_{\phi^{-1}(6)})\leq 2$ by Theorem \ref{Steenrod_alg_chromatic_number}. 
 Since $K^1_{\phi^{-1}(6)}$ is $K'_{\phi^{-1}(6)}$ together with some disjoint vertices, which we can colour freely, 
 $\chi(K^1_{\phi^{-1}(6)})\leq 2$. 
 \end{proof}

\section{Some Problems}\label{sec:future_problem}
We propose some questions related to span colouring and Steenrod's problem. We consider 
Questions \ref{mostspan} and \ref{mostreal} the most interesting. 
\subsection{Chromatic number questions}
\begin{question}
What is the relationship between $s_{\mathbf{k}}\chi(G)$ and $\chi(G)$?
\end{question}
We also suggest some more specific cases of this question. 
\begin{question}
What is $\chi(A_{\mathbf{k}^n})$? 
\end{question}
The simplest case is $\chi(A_{(\mathbf{\zZ/2})^3})$ which is at least $4$ and should not be too hard to resolve. 
We guess that $\chi(A_{(\mathbf{\zZ/2})^3})=4 \ {\mbox or\ } 5$.  We know $\chi(A_{\mathbf{k}^2})=2$ and 
Propositions \ref{cong}, \ref{chromatic_An} and \ref{chromatic_difference}
show in many cases 
that $\chi(A_{\mathbf{k}^n})>n$. We ask if this is true more generally. 
\begin{question}
If $n>2$, is $\chi(A_{\mathbf{k}^n})>n$?
\end{question}
Our proofs use counting of vertices, one could also count edges or larger cliques. This may solve more cases, but seems unlikely to be enough of all cases even for finite fields. Perhaps there are more subtle ways of show such inequalities. 
We ask two other related questions. 

\begin{question} 
What is the smallest graph for which $s_{\mathbf{k}}\chi(G)<\chi(G)$?
\end{question}

We know we can take $G=A_{(\zZ/2)^3}$ to get an inequality, can be still get an inequality for subgraphs of 
$A_{(\zZ/2)^3}$?

\begin{question}
How large can $\chi(G)-s_{\mathbf{k}}\chi(G)$ be?  
\end{question}

This difference is at least as big as $n-\chi(A_{\mathbf{k}^n})$, relating this question to the first one. Can the difference be bigger? 

\begin{question}
Is there a nice class of graphs for which $\chi(G)=s_{\mathbf{k}}(G)$?
\end{question}

We saw in Proposition \ref{important2} that $clique(G)\leq s_{\mathbf{k}}(G)$, so graphs with 
$clique(G)=\chi(G)$ would be  such a class. Is there another already well known class such that 
$\chi(G)=s_{\mathbf{k}}(G)$?

We want to define an integral version of span colouring. 
Define $A_{\mathbb{Z}^n}$ to be the graph with 
\[
   V(A_{\mathbb{Z}^n})=\{(U,V)| U, V\subset \mathbb{Z}^n, U\oplus V=\mathbb{Z}^n, \text{rank } U=1\}
\] 
and 
\[
   E(A_{\mathbb{Z}^n})=\{((U,V), (U',V'))\in V(A_{\mathbb{Z}^n})^2| U\subset V', U'\subset V\}.
\]

Next an $n$-span colouring of $G$ over $\zZ$ is a map of graphs $G\rightarrow A_{\mathbb{Z}^n}$. 
The chromatic number of this colouring is denoted $s_{\mathbb{Z}}\chi$.

Instead of the condition $U\oplus V=\zZ^n$ we could have the condition $U\cap V=0$. 
In the field case this is the same, but if we use it for $\zZ$ the following proof that colourings localize would not work. 
We use the conventions $\zZ/0=\zQ$ and $s_0\chi=s_{\zQ}\chi$. 

\begin{proposition}\label{compare}
Suppose $G$ is a graph and $\mathbf{k}$ is a field. 
\begin{enumerate}
\item There is a map $A_{\mathbb{Z}^n}\rightarrow A_{\mathbf{k}^n}$ and hence $s_{\mathbb{Z}}\chi\geq s_\mathbf{k}\chi(G)$. 
\item If $char(\mathbf{k})=p$ then there is a map $A_{(\mathbb{Z}/p)^n}\rightarrow A_{\mathbf{k}^n}$ 
and hence $s_p\chi\geq s_\mathbf{k}\chi(G)$. 

\item The is a map $K_n\rightarrow A_{\zZ^n}$ and hence $\chi(G)\geq s_{\zZ}\chi$. 
\end{enumerate}
\end{proposition}

\begin{proof}
First identify $(\zZ)\otimes \mathbf{k}$ with $\mathbf{k}$.
Let $(U,V)\in V(A_{\zZ^n})$ then since 
$(U\oplus V)\otimes \mathbf{k}\cong U\otimes \mathbf{k}\oplus V \otimes \mathbf{k}$ we get
$(U\otimes \mathbf{k}, V\otimes \mathbf{k})\in A_{\mathbf{k}^n}$. Let 
$\rho\colon A_{\mathbb{Z}^n}\rightarrow A_{\mathbf{k}^n}$ be the corresponding map of graphs. 
Note that if $((U,V), (U',V'))\in V(A_{\mathbb{Z}^n})^2$ is an edge then 
we have a retraction of $U\subset V'$ given by the composition 
$U\subset V'\subset \mathbb{Z}^n=U\oplus V \rightarrow U$ 
where the last map is the projection. This implies that $U\otimes \mathbf{k} \subset V'\otimes \mathbf{k}$ and similarly 
$U'\otimes \mathbf{k} \subset V\otimes \mathbf{k}$. So $\rho$ preserves edges and is a map of graphs. 
This completes the proof of 
(1). The proof of (2) is similar. 

For a basis $\{x_1,\dots x_n\}$ of $\mathbb{Z}^n$, the vertices in $A_{\mathbb{Z}^n}$, $(\langle x_i\rangle, \langle x_1,\dots x_{i-1},x_{i+1},\dots x_n\rangle)$ form a subgraph $K_n\subset A_{\mathbb{Z}^n}$ proving (3). 
\end{proof}

\begin{question}\label{mostspan}
Proposition \ref{compare} implies that 

$$\chi(G)\geq s_{\zZ}\chi(G)\geq \sup_{p \text{ prime or } 0}\{ s_{p}\chi(G)\}.$$ 

Are these inequalities equalities?
\end{question}

Both equalities holding in general seems too good to be true, and would give a way of computing chromatic number by computing some (algebraically) local versions of it where scheme theoretic techniques 
might be useful.

\begin{question}
Suppose $char(\mathbf{k})=char(\mathbf{k}')$, does this imply that 
$s_{\mathbf{k}}(G)=s_{\mathbf{k}'}(G)$?
\end{question}

Using the argument in Proposition \ref{compare} if $\mathbf{k}'$ is a field extension of $\mathbf{k}$ 
then $s_{\mathbf{k}}(G)\geq s_{\mathbf{k}'}(G)$. 
From Example \ref{God} we know that $s_{\mathbb{R}}\chi(S^2)=3$, while 
$\chi(S^2)=4$, so a special case of the last question which could be a possible counterexample would be:

\begin{question}
Is $s_{\zQ}\chi(S^2)=3 \text{ or } 4$?
\end{question}

%

We can generalize the span colouring to coloured by poset.
Let $P$ be a poset.
We assume that $P$ has a least element $o$. For any $p,q\in P$ let   $p\vee q$ denote the least upper bound if it exists.
Since $(p\vee q)\vee r=p\vee (q\vee r)$ for $p,q,r\in P$, for a finite subset $S\subset P$ we can define the least upper bound of $S$, $\bigvee S$, if $S$ is empty then $\bigvee S=o$.
We take a subset $Q\subset P$.
Then for a graph $G$, the map $f\colon V(G)\rightarrow Q$ is $(P,Q)$-colouring of a graph $G$ when $f$ satisfies the following conditions:
\begin{itemize}
\item  $\bigvee f(N(v))$ exists. 
   \item For any $v$, the only lower bound of $f(v)$ and $\bigvee f(N(v))$ is $o$.
\end{itemize} 
This $(P,Q)$-colouring contains the ordinary and span colouring.
\begin{example}
   \begin{enumerate}
      \item When $P=2^{[n]}$ and $Q=\{\{1\},\dots \{n\}\}$, a $(P,Q)$-colouring is an ordinary colouring.
      \item Let $P$ be the set of linear subspaces of $\mathbf{k}^n$, and $Q$ be the poset of linear subspaces of $\mathbf{k}^n$ with dimension $1$.
      Then a $(P,Q)$-colouring is an $n$-intermediate span colouring.
      \item When $P = \{A \in 2^{[n]} | |A| \geq k\} \cup \{o\}$ and $Q = \{A \in 2^{[n]} | |A|=k\}$, a $(P, Q)$-colouring
corresponds to a graph homomorphism from $G$ to Kneser graph $KG_{n,k}$.
      \item For $s_{\zZ}$ we can use $P$ the submodules of $\zZ^n$ whose inclusions are split and $Q$ 
      being those with rank $1$. As in case $(2)$ these are not quite the same as maps into  $A_{\zZ^n}$.
 Note that a submodule generated by split submodules need not be split but lies inside a unique smallest split submodule. Something similar could be done for modules over other rings. 
   \end{enumerate}

\end{example}

\subsection{Realizability questions}
This paper only looks at realizability of $\zZ$-algebras but the problem can also be studied over $\zZ/p$ and $\zZ_{(p)}$. 

Theorem \ref{topboundsfromgraphs} says
for any graph $G$, $s_2\chi(G)\leq \chi_{Top}(G)\leq \chi(G)$. 
\begin{question}\label{mostreal}
What is $\chi_{Top}(G)$? Which of the inequalities above can be strict? Is there a combinatorial or purely graph theoretic description of $\chi_{Top}(G)$?
\end{question}
Since $s_2\chi(A_{(\zZ/2)^3} )=3<\chi(A_{(\zZ/2)^3} )$ at least of of the two is strict.  
For a graph with $s_2\chi(G)<\chi(G)$ to show 
$\chi_{Top}(G)=s_2\chi(G)$ probably would require a new construction that builds a space out of a span colouring. As mentioned in the introduction one way to do this would be to upgrade the diagram of $\SA_2$ algebras to one of spaces, this is made more plausible by the fact that the height of the poset $P_{\max}(\Delta^{n-1}\ast G)$ is at most $2$. 
Showing $\chi_{Top}(G)<s_2\chi(G)$ might be helped by looking at more refined invariants such as secondary operations, or operations in other cohomology theories. Maybe we can even still get more from  relations in the primary operations. Out of this information we would need to construct a graph invariant. 
The easiest unresolved case of the general question is:
\begin{question}
Is $A(n,A_{(\zZ/2)^3})$ realizable? 
\end{question}

Not only is $A_{(\zZ/2)^3}$ the smallest graph we know with $s_2\chi(G)<\chi(G)$, the graph has a large symmetry group giving a tool that might help in a construction. 
For a graded algebra $A$ we let $A^k$ denote its $k$-skeleton. That is $A^k=A/(\text{elements of degree}>k)$. Note that we can see using topological skeleta that $A$ realizable also implies that $A^k$ is realizable for every $k$, leading to:
\begin{question}
Is $A(n,A_{(\zZ/2)^3})^{18}$ realizable?
\end{question}

Generalizing this question and localizing at $2$ we ask:

\begin{question}
Is $A(n,G)^{18}\otimes \zZ_{(2)}$ realizable if and only if $s_2\chi(G)\leq n$?
\end{question}
If $A(n,G)\otimes \zZ/2$ has an $\SA_2$ action then the 
$A(n,G)^k\otimes \zZ/2$ have induced actions of Steenrod and all the proofs in 
Section \ref{sec:Span_colouring_and_Steenrod_algebra} only use equations in degrees at most 18 so
the proof of Corollary \ref{realimpspan} should carry over to showing that 
$A(n,G)^{18}\otimes \zZ_{(2)}$ realizable implies  $s_2\chi(G)\leq n$. Starting with $s_2\chi(G)\leq n$ getting a space whose cohomology with coefficients in  $\zZ_{(2)}$ is  $A(n,G)^{18}\otimes \zZ_{(2)}$ may be easier with the dimension restriction. 

\subsection{Stabilization, Fractional colourings and Cores}

We stabilize the ambient vector space to get sequences other graphs colourings and chromatic numbers. We don't know if this gives any new chromatic numbers. We also consider how out graphs are related to cores (which a topologist would call a minimal retract). 

\begin{definition}
A graph $G$ is a {\em core} if every endomorphism of $G$ is an automorphism. 

A graph $i\colon H\subset G$ is a {\em retract} if there is a map $r\colon G\rightarrow H$ such that $r\circ i$ is the identity on $H$. 
\end{definition}
So we can see that cores are the same as minimal retracts.

We fix a field and write $A_{\mathbf{k}^n}$ as $A_n$. 
The (vertices of) the graphs $A_{n}$ are the complement of the partial flag variety $(V_1,V_2)$ in (lines,dim n-1 subspaces). Here $V_1$ is a 1-dimensional subspace and $V_2$ is $n-1$ dimensional. We can also define a fractional colouring version by letting $V_1$ be $k$-dimensional and $V_2$, $n-k$ dimensional, and we get corresponding graphs
$A_{n,k}$. With this notation $A_{n,1}=A_n$. Both of these can be stabilized. In other words in the original definition the subspaces are in $k^n$, but we can take $(U,V)$ to be $k$ and $n-k$ dimensional subspaces in $\mathbf{k}^{n+l}$. Correspondingly we get graphs $A_{n,k,l}$ and we can ask if these cores. For each fixed $n,k$ there are at most finitely many $l$ for which they could be cores. Since 
$\chi(A_{(\zZ/2)^3})>3$ is seems plausible that $A_n=A_{n, 1, 0}$ is a core. We also get chromatic numbers
$\chi_{n,k}$ and $\chi_{n,k,l}$. Since $k^{n+l}\subset k^{n+l+1}$, we see that  $\chi_{n,k,l}\leq \chi_{n,k,l'}$, for $l\geq l'$ so these numbers stabilize as $l$ increases.  Are they already stable for $l=0$? In other words is the inequality
$\chi_{n,k}=\chi_{n,k,0}\leq \lim_{l\rightarrow \infty} \chi_{n,k,l}$ always an equality? This would mesh with none of the $A_{n,k.l}$ being cores for $l>0$. Note that these can all be phrased in terms of poset colourings, $A_{n,k,l}$ corresponds to $P=\{ \mbox{subpaces of } {\mathbf{k}^{n+l}}\}$ and $Q=\{ k \mbox{ dimensional subpaces of } {\mathbf{k}^{n+l}}\}$, or for the set version $P=2^{[n+l]}$ and $Q=\{ \mbox{subsets of } [n+l] \mbox{ of size } k\}$. 

We use a slightly different stabilization for $\chi_{Top}$. If our algebra 
$A(n,G)=SR(\Delta^{n-1}\ast G)$ is realizable 
by a space $X$ then $X\times BSU(3)$ will realize $A(n,G\ast \{y\})=SR(\Delta^n\ast G\ast \{y\})$. So we can stabilize by 
letting $\chi_{Top, t}(G)=\chi_{Top}(G\ast \{ y_{1} \} \ast \dots \ast \{ y_{t}\})-t$. Since $\chi(G\ast \{y\})=\chi(G)+1$
and $s_2\chi(G\ast \{y\})= s_2\chi(G)+1$ 
we get that 
$$s_2\chi(G)\leq \chi_{Top, t}(G)\leq \chi_{Top}(G)\leq \chi(G).$$ Which of these inequalities can be strict? More specifically:

\begin{question}
Does $\chi_{Top, t}(G)=\chi_{Top}(G)$ for all $G$ and $t$?
\end{question}

\bibliographystyle{amsalpha}

\end{document}